\documentclass[11pt, twoside]{article}

\usepackage{graphicx}
\usepackage[caption=false]{subfig}
\usepackage{amsmath}
\usepackage{amssymb,amsfonts}
\usepackage{amsthm}
\usepackage{bm}
\usepackage{mathrsfs}
\usepackage{amssymb}
\usepackage{multirow}
\usepackage[most]{tcolorbox}
\usepackage[margin=1in]{geometry}
\usepackage{algorithm}
\usepackage{algpseudocode}
\numberwithin{equation}{section}
\usepackage{multirow}
\usepackage{makecell}
\usepackage{booktabs}
\theoremstyle{definition}
\newtheorem{theorem}{Theorem}

\newtheorem{lemma}{Lemma}

\newtheorem{definition}{Definition}

\newtheorem{example}{Example}
\newtheorem{assumption}{Assumption}

\usepackage{cite}
\usepackage{hyperref}
\usepackage[nameinlink]{cleveref}
\newcommand{\vertiii}[1]{{\left\vert\kern-0.25ex\left\vert\kern-0.25ex\left\vert #1 
		\right\vert\kern-0.25ex\right\vert\kern-0.25ex\right\vert}}

\usepackage{fancyhdr}
\begin{document}
	
	\title{A new reduced order model of  linear parabolic PDEs}

\author{Noel Walkington 
	\thanks{Department of Mathematics Science, Carnegie Mellon University, Pittsburgh, PA, USA (\mbox{noelw@andrew.cmu.edu},\quad \mbox{yangwenz@andrew.cmu.edu}).}
	\and  Franziska Weber 	\thanks{Department of Mathematics, university of california Berkeley, Berkeley, CA, USA (\mbox{fweber@math.berkeley.edu}).}
	\and
	Yangwen Zhang		\footnotemark[1] 
}

\date{\today}

\maketitle

\begin{abstract}
How to build an accurate reduced order model (ROM) for  multidimensional time dependent partial differential equations (PDEs) is quite open. In this paper, we propose a new ROM for linear parabolic PDEs.  We prove that our new method can be orders of magnitude faster than
standard solvers, and is also much less memory intensive. Under some assumptions on the problem data, we prove that the convergence rates of the new method is the same with standard solvers. Numerical experiments are presented to confirm our theoretical result.
\end{abstract}

\section{Introduction}
It is well known that the computational cost of solving multidimensional time dependent partial differential equations (PDEs) can be extremely high. Many model order reduction  (MOR) methods  have been proposed to reduce the computational cost.  Proper orthogonal decomposition (POD) methods are  widely used  in both academic and industry  \cite{MR3419868}. 

There are extensive studies to prove the error bounds of the POD-ROM \cite{Singler_New_SINUM_2014,Kunisch_Volkwein_NM_2001,Chapelle_Gariah_Sainte_Galerkin_M2AN_2013,MR3987425,MR4172732,MR4296762,Kunisch_Volkwein_SINUM_2002} under the assumption that the data in the POD-ROM are not varied from the original PDE model. However, when data is changed in the POD-ROM, the solution of the POD-ROM can be either surprisingly accurate or completely unrelated to the solution of the full order model (FOM). It is worthwhile mentioning that in \cite{MR2327057}, the authors studied the effects of small perturbations in the POD-ROM. They  explained why in some applications this sensitivity is a concern while in others it is not. Therefore, hybrid methods (FOM/ROM switched back and forth) are commonly used in practice \cite{MR2595807,MR2873253,MR3697034,bai2020deim}. However, the savings are not remarkable.

How to build an accurate and efficient ROM for time dependent PDEs is still an open problem.
As a first step toward this open problem, we propose a new ROM for linear parabolic PDEs.   Next, we briefly describe the standard finite element method (FEM)  and our algorithm for solving the following heat equation:
\begin{align}\label{heat}
	\begin{split}
		u_t - \Delta u & = f(x) \quad \;\;\textup{in}\; \Omega\times (0,T],\\
		u & = 0  \qquad \quad \textup{on}\; \partial \Omega\times (0,T],\\
		u(\cdot, 0)& = 0 \qquad \quad \textup{in}\; \Omega. 
	\end{split}
\end{align}

Let $M$ and $A$ be the mass  and stiffness matrices, respectively, and  $b$ is the load vector. Here we assumed that the right hand side $f$ is independent of time. For general data, see \Cref{NEwway}. 	 The semi-discrete of \eqref{heat} is to find $\alpha_h(t)\in \mathbb R^N$ satisfying
\begin{align}\label{semi-d}
	\begin{split}
		M \alpha_h'(t) + A \alpha_h(t) &= b,\quad t\in (0,T],\\
		\alpha_h(0)& = 0,
	\end{split}
\end{align}
where $\alpha_h(t)$ is the coefficient of the FEM solution under the FEM basis functions. Implicit solvers are commonly used to solve the ODE system \eqref{semi-d}. 	The computational cost can be huge (see \Cref{table_0}) for solving this system.  Next, we introduce our  new algorithm.
\begin{itemize}
	\item [(1)]  We generate a Krylov sequence for the problem data;  see lines 1-4 in \Cref{Brrief}. This is different from the standard POD method. The standard POD method  uses the solution data to generate a reduced basis, which is the drawback of the  method---relies too heavily on the solution data. The parameter $\ell$ we choose is a small integer, usually less than $10$. This is the main computational cost in our new algorithm. To adaptively choose $\ell$, see \Cref{D_matrixalgorthm2}.
	
	\vspace{0.15cm}
	
	\item [(2)] We use the idea from POD to find $r$ optimal reduced basis functions, the span of these  $r$ functions generates a dimension $r$ space; see  lines 5-9 in \Cref{Brrief}\footnote{We adopt Matlab notation herein:  $\Psi(:, 1:r)$ be the first $r$ columns of $\Psi$;   $\Lambda(1:r,1:r)$ be the $r-$th leading principal minor of $\Lambda$;  \texttt{eig} is the Matlab built-in function to find  the eigenvectors and eigenvalues of a matrix.}. 
	
	\vspace{0.15cm}
	
	\item[(3)] We project the heat equation \eqref{heat} onto the above reduced space, and obtain the reduced mass matrix $M_r$,  the reduced stiffness matrix $A_r$ and the reduced load vector $b_r$; see line 10 in \Cref{Brrief}. Then the semi-discrete of the ROM is to find $\alpha_r(t)\in \mathbb R^r$ satisfying
	\begin{align}\label{ROM111}
		\begin{split}
			M_r  \alpha_r'(t) + A_r \alpha_r (t) &= b_r, \qquad t\in (0,T],\\
			\alpha_r(0)& = 0.
		\end{split}
	\end{align}
	
	\vspace{0.15cm}
	
	\item[(4)] For the time integration,  we choose  the backward Euler for the first step and then apply the two-steps backward differentiation formula (BDF2); see lines 11-14 in \Cref{Brrief}. 
	
	\vspace{0.15cm}
	
	\item[(5)] We then return to the FOM. We note that $Q\alpha_r^n$ is the coefficient of the  solution at time $t_n$ under the standard FEM basis functions.   
\end{itemize}
\begin{algorithm}[H]
	\caption{Our new algorithm for solving the heat equation \eqref{heat}}
	\label{Brrief}
	{\bf{Input}:}   $\ell$, $M$, $A$, $b$, \texttt{tol}
	\begin{algorithmic}[1]
		\State Solve $A\texttt{u}_{1} = b$;
		\For{$i=2$ to $\ell$}
		\State Solve $A\texttt{u}_{i} = M \texttt{u}_{i-1}$;		
		\EndFor
		\State Set $U_\ell = [\texttt{u}_{1} \mid \texttt{u}_{2} \mid \ldots\mid \texttt{u}_{\ell}]\in \mathbb R^{N\times \ell}$;
		\State Set $K_\ell = U_\ell^\top A U_\ell\in \mathbb R^{\ell\times \ell}$;
		\State $[\Psi, \Lambda] = \texttt{eig}(K_\ell)$;
		\State Find minimal $r$ such that $\sum_{i=1}^r{\Lambda(i,i)}/\sum_{i=1}^\ell{\Lambda(i,i)}\ge 1- \texttt{tol}$;
		\State Set $Q = U_\ell \Psi(:,1:r)(\Lambda(1:r;1:r))^{-1/2}$;
		\State Set $M_r = Q^\top M Q$; $A_r = Q^\top A Q$; $b_r = Q^\top b$;
		\State Solve $\left(\dfrac{1}{\Delta t}M_{r}+ A_r\right) \alpha_r^1= b_{r}$;
		\For {$n = 2$ to $N_T$}	
		\State Solve $\left(\dfrac{3}{2\Delta t}M_{r}+  A_r\right) \alpha_r^n=  \dfrac{1}{\Delta t}M_{r}\left(2\alpha_r^{n-1} - \dfrac 1 2 \alpha_r^{n-2}\right)+ b_{r}$;
		\EndFor
		\State {\bf return} $Q, \alpha_r^n$
	\end{algorithmic}
\end{algorithm}
Numerical result in \Cref{table_1} shows that our new ROM for the linear parabolic PDEs is accurate. Comparing with \Cref{table_0}, we see that our new algorithm is  efficient, and the savings are remarkable.

In \Cref{Error_analysis}, we prove that the convergences rates of our new algorithm is  the same with the standard FEM under an assumption on the problem data; see \Cref{Main_res}.  Numerical experiments in \Cref{NEwway} are presented to confirm our theoretical result even when the assumption is not satisfied. 

Moreover,  we show that the singular values of the Krylov sequence are exponential decaying; see \Cref{UAU}. This guarantees that the dimension of our ROM is extremely low. To the best of our knowledge, this is one of the first theoretical result in MOR.

\section{The new algorithm and its implementation}\label{Wei_POD}

Throughout the paper, we assume $\Omega\subset \mathbb R^d$, $d=1, 2,3$, and when $d\geq 2$, $\Omega$ is a bounded polyhedral domain. Let $V$ and $H$ be real separable Hilbert spaces and suppose that $V$ is dense in $H$ with compact embedding. By $(\cdot, \cdot)$ and $\|\cdot \|$ we denote the inner product and norm  in $H$. The inner product in $V$ is given by a symmetric bounded, coercive, bilinear form $a: V \times V \rightarrow \mathbb{R}$ :
$$
(\varphi, \psi)_{V}=a(\varphi, \psi) \quad \text { for all } \varphi, \psi \in V
$$
with an associated norm given by $\|\cdot\|_{V}=\sqrt{a(\cdot, \cdot)}$. Since $V$ is continuously embedded into $H$, there exists a constant $C_V>0$ such that
\begin{align}\label{embedding}
	\|\varphi\| \leq  C_V\|\varphi\|_{V} \quad \text { for all } \varphi \in V.
\end{align}

Let  $\langle\cdot, \cdot\rangle_{V^{\prime}, V}$ denotes the duality pairing between $V$ and its dual $V'$.  By identifying $H$ and its dual $H^{\prime}$ it follows that
$$
V \hookrightarrow H=H^{\prime} \hookrightarrow V^{\prime},
$$
each embedding being continuous and dense.

For given $f \in L^{2}(0, T ; H)$ and $u_{0} \in H$ we consider the linear parabolic problem:
\begin{align}\label{Parabolic_Original_Model}
	\begin{split}
		\frac{{\rm d}}{{\rm d} t}( u(t), v)+a(u(t), v)&= (f, v)	\qquad\;  \forall v\in V,\\
		(u(0), v) &=  (u_0, v) \qquad  \forall v\in V.	
	\end{split}
\end{align}

\subsection{Finite element method (FEM)} Let $\mathcal T_h$ be a collection of disjoint shape regular simplices that partition $\Omega$. The functions $\varphi_{1}, \ldots, \varphi_{N}$ denote $N$ linearly independent nodal basis functions. On each element $K\in \mathcal T_h$, $\varphi_i|_K\in \mathcal P^k(K)$, where $ \mathcal P^k(K)$ denotes the set of polynomials of degree at most $k$ on the element $K$. Then we define the $N$-dimensional subspace:
$$
V_h=\operatorname{span}\left\{\varphi_{1}, \ldots, \varphi_{N}\right\} \subset V.
$$

To simplify the presentation, we assume the right hand side (RHS) $f$ is independent of time and the initial condition $u_0=0$. We shall discuss the general case ($f$ depends on both time and space and nonzero initial condition) in \Cref{NEwway}.
First, we consider the semi-discretization of \eqref{Parabolic_Original_Model}, i.e., find $u_h(t) \in C((0, T]; V_{h})$   satisfying
\begin{align}\label{Weak_heat}
	\begin{split}
		\frac{\mathrm{d}}{\mathrm{d} t}\left( u_{h}(t), v_h\right)+ a( u_h(t),  v_h) &= (f, v_h),  \quad\quad\quad\quad \;\forall v_h\in V_h,\\
		(u_h(0), v_h)&=0,  \qquad\qquad\quad\quad \;\;\forall v_h\in V_h.
	\end{split}
\end{align}
Next, let $0=t_0<t_1<\ldots<t_{N_T}=T$ be a given grid in $[0, T]$ with equally step size $\Delta t$. To solve \eqref{Weak_heat} we apply the backward Euler for the first step and then apply the two-steps backward differentiation formula (BDF2). Specifically, we find $u_h^n\in V_h$ satisfying
\begin{align}\label{discreteODE}
	\begin{split}
		\left( \partial_t^+ u_h^n, v_h\right)+ a\left(u_h^n, v_h\right) &= (f, v_h), \quad \forall v_h\in V_h,\\
		(u_h^0, v_h) &= 0, \qquad\quad\; \forall v_h\in V_h,
	\end{split}
\end{align}
where
\begin{align*}
	\partial_t^+ u_h^n = \begin{cases}
		\dfrac{u_h^n - u_h^{n-1}}{\Delta t},\qquad\qquad\quad\;  n=1,\\[0.4cm]
		\dfrac{3u_h^n - 4u_h^{n-1} + u_h^{n-2}}{2\Delta t},\quad n\ge 2.
	\end{cases}
\end{align*}
The computational cost can be very high if the mesh size $h$ and time step $\Delta t$ are small. 
\begin{example}\label{Example0}
	In  this example, let  $\Omega = (0,1)\times (0,1)$, we consider the equation \eqref{Parabolic_Original_Model} with
	\begin{align*}
		a(u,v) = (\nabla u, \nabla v), \qquad u_0 = 0, \qquad f = 10^4  (x-0.1)(y-0.2)(x-0.3)(y-0.4).	\end{align*}
	We use linear finite elements for the spatial discretization, and  for the time discretization, we use the backward Euler for the first step and then apply the two-steps backward differentiation formula ({BDF2}) with time step $\Delta t = h$, here $h$ is	the mesh size (max diameter of the triangles in the mesh). 		 For solving linear systems, we apply the Matlab built-in solver backslash ($\backslash$) in 2D  and algebraic multigrid methods \cite{MR972756} in 3D. We report the wall time\footnote{All the code for all examples in the paper has been created by the authors using Matlab R2020b
		and has been run on a laptop with MacBook Pro, 2.3 Ghz8-Core Intel Core i9 with 64GB 2667 Mhz DDR4. We use the Matlab built-in function \texttt{tic}-\texttt{toc} to denote  the real simulation time.}  in \Cref{table_0}. 
	\begin{table}[H]
		\centering
		{
			\begin{tabular}{c|c|c|c|c|c|c|c}
				\Xhline{1pt}
				
				\cline{3-8}
				$h$	& $1/2^4$ &$1/2^5$ 
				&$1/2^6$ 
				&$1/2^7$  &$1/2^8$ &$1/2^9$  &$1/2^{10}$ 
				\\
				\cline{1-8}
				{Wall time (s)}
				& 	  0.130&   0.072&   0.535&   4.456&   33.75&   340.1&   3754\\ 
				\Xhline{1pt}

			\end{tabular}
		}
		\caption{ The wall time (seconds) for the simulation of \Cref{Example0}.}\label{table_0}
	\end{table}		
\end{example}

\subsection{Reduced basis generation} 
Recall that the source term $f$ does not depend on time.  Let $\ell$ be a small integer and $\mathfrak{u}_h^0 = f$.  For $1\le i \le \ell$,  we find $\mathfrak{u}_h^i\in V_h$ satisfying
\begin{align}\label{Generate_basis}
	\begin{split}
		a\left( \mathfrak{u}_h^i,  v_{h}\right) 
		=\left(\mathfrak{u}_h^{i-1}, v_{h}\right) \qquad  \forall v_{h} \in V_{h}.
	\end{split}
\end{align}

To find an optimal reduced basis, we then consider the following minimization problem:
\begin{gather}\tag{P1}\label{P1}
	\begin{split}
		\min _{{\widetilde \varphi}_{1}, \ldots, {\widetilde\varphi}_{r} \in V_h}  \sum_{j=1}^{\ell}  \left\|\mathfrak{u}_h^j-\sum_{i=1}^{r}\left( \mathfrak{u}_h^j, {\widetilde\varphi}_{i}\right)_{V} {\widetilde\varphi}_{i}\right\|_{V}^{2} \quad \text { s.t. }\left({\widetilde\varphi}_{i}, {\widetilde\varphi}_{j}\right)_V=\delta_{i j},\quad  1 \leq i, j \leq r. 	
	\end{split}
\end{gather}

The proof of the following lemma can be found in \cite[Theorem 2.7]{gubisch2017proper}.
\begin{lemma}\label{Optimization}
	The solution to  problem \eqref{P1} is given by the first $r$ eigenvectors  of $\mathcal{R}: V_h\to V_h$:
	\begin{align*}
		\mathcal{R} \varphi= \sum_{j=1}^{\ell}\left( \mathfrak{u}_h^j, \varphi\right)_{V} \mathfrak{u}_h^j  \quad \textup  { for } \varphi \in V_h. 
	\end{align*}
	Furthermore, let $\{\lambda_i({\mathcal R})\}_{i=1}^\ell$ be the first $\ell$ eigenvalues of $\mathcal R$, then we have 
	\begin{align*}
		\textup{argmin\eqref{P1}}= \sum_{i=r+1}^\ell \lambda_i({\mathcal R}).
	\end{align*}		
\end{lemma}

\Cref{Optimization} is not  practical since   $\mathcal R$ is an abstract operator. We introduce the following lemma.
\begin{lemma}\label{Eig}
	Let $\mathcal U:  X \to Y$ be a compact linear operator, where $X$ and $Y$ are separable Hilbert spaces and  $\mathcal U^{*}: Y \rightarrow X$ is the Hilbert adjoint operator. The the nonzero (positive) eigenvalues $\left\{\lambda_{k}\right\}$ of $\mathcal U^{*} \mathcal U$ and $\mathcal U \mathcal U^{*}$ are the same, and if $x_k$ is an orthonormal eigenfunction of $\mathcal U^{*}\mathcal U$, then
	\begin{align}\label{eiU}
		y_k = \frac{1}{\sqrt{\lambda_{k}}} \mathcal U x_k
	\end{align}
	are  orthonormal eigenfunctions of $ \mathcal U\mathcal U^{*}$.
\end{lemma}

\Cref{Eig} gives us an alternative way to compute eigenpair of $\mathcal R$: if we can not find eigenpair of $\mathcal R$  and it can be  rewritten as $\mathcal U\mathcal U^{*}$; then we should consider to compute the eigenpair  of $\mathcal U^{*}\mathcal U$.

Let us define the linear and bounded operator $\mathcal U: \mathbb R^\ell \to V_h$ by
\begin{align}\label{Ualpha}
	\mathcal U \alpha = \sum_{i=1}^\ell \alpha_i \mathfrak{u}_h^i,\qquad \textup{where}\; \alpha = [\alpha_1,\alpha_2,\ldots,\alpha_{\ell}]^\top\in \mathbb R^\ell.
\end{align}
The Hilbert adjoint $\mathcal U^*: V_h\to \mathbb R^\ell$ satisfies
\begin{align*}
	\left( \mathcal U^* v_h, \alpha\right)_{\mathbb R^\ell} = 	\left(v_h,  \mathcal U \alpha\right)_V = \sum_{i=1}^\ell \alpha_i (\mathfrak{u}_h^i, v_h)_V.
\end{align*}
This implies
\begin{align}\label{UStarvh}
	\mathcal U^* v_h = \left[(\mathfrak{u}_h^1, v_h)_V,\ldots, (\mathfrak{u}_h^\ell, v_h)_V \right]^\top.
\end{align}
Then we have 
\begin{align*}
	\mathcal U\mathcal U^*  v_h = \sum_{i=1}^\ell  (\mathfrak{u}_h^i, v_h)_V\mathfrak{u}_h^i = \mathcal R v_h.
\end{align*}

Motivated by \Cref{Eig}, we shall compute $\mathcal U^*	\mathcal U$, it is obvious that $\mathcal U^*\mathcal U: \mathbb R^\ell \to \mathbb R^\ell$, i.e., the operator $\mathcal U^*\mathcal U$ is a matrix, we can  compute its eigenvectors and eigenvalues in practice. Then we can get  eigenvectors   of 
$\mathcal R$ by using \Cref{Eig}. 
\subsubsection{Computation of $\mathcal U^*\mathcal U$} We first define the mass matrix $M$, the stiffness matrix $A$ and the load vector $b$ by
\begin{align}\label{MassStiffnessload}
	M_{i j}=(\varphi_{j}, \varphi_{i}), \quad     A_{i j}= a(\varphi_{j}, \varphi_{i}), \quad  \quad b_i = (f, \varphi_i).
\end{align}
Let $\texttt{u}_{i}\in \mathbb R^N$ be the coefficient of $\mathfrak{u}_h^i$ under the finite element basis $\{\varphi_j\}_{j=1}^N$, $1\le i\le \ell$,  i.e., 
\begin{align}\label{Coefficients}
	\mathfrak{u}_h^i = \sum_{j=1}^{N} \left(\texttt{u}_{i}\right)_j \varphi_j,
\end{align}
where $(\alpha)_j$ denotes the $j$-th component of the vector $\alpha$. Then substituting  \eqref{Coefficients} into \eqref{Generate_basis} we obtain
\begin{align}\label{ObtaiNp}
	A\texttt{u}_i = M\texttt{u}_{i-1},\qquad 2\le i\le \ell, \quad \textup{with} \quad A\texttt{u}_{1} = b.
\end{align}
Next, we define matrix $U_\ell$ by
\begin{align}\label{Coefficients_Matrix}
	U_\ell = [\texttt{u}_{1} \mid \texttt{u}_{2} \mid \ldots\mid \texttt{u}_{\ell}]\in \mathbb R^{N\times \ell}.
\end{align}

By \eqref{Ualpha} and \eqref{UStarvh} we can get 
\begin{align*}
	\mathcal U^*\mathcal U  \alpha  &=  \left[\left(\mathfrak{u}_h^1, \sum_{i=1}^\ell \alpha_i \mathfrak{u}_h^i\right)_V,\ldots, \left(\mathfrak{u}_h^\ell, \sum_{i=1}^\ell \alpha_i \mathfrak{u}_h^i\right)_V \right]^\top := K_\ell\alpha,
\end{align*}
where 
\begin{align}\label{K_matrix}
	\begin{split}
		K_\ell  &= \left[\begin{array}{cccc}
			(\mathfrak{u}_h^1, \mathfrak{u}_h^1)_V&	(\mathfrak{u}_h^1, \mathfrak{u}_h^2)_V&\cdots&	(\mathfrak{u}_h^1, \mathfrak{u}_h^\ell)_V\\[0.2cm]
			(\mathfrak{u}_h^2, \mathfrak{u}_h^1)_V&	(\mathfrak{u}_h^2, \mathfrak{u}_h^2)_V&\cdots&	(\mathfrak{u}_h^2, \mathfrak{u}_h^\ell)_V\\
			\vdots&\vdots&\cdots&\vdots\\
			(\mathfrak{u}_h^\ell, \mathfrak{u}_h^1)_V&	(\mathfrak{u}_h^\ell, \mathfrak{u}_h^2)_V&\cdots&	(\mathfrak{u}_h^\ell, \mathfrak{u}_h^\ell)_V\\
		\end{array}\right] 
		= \left[\begin{array}{cccc}
			\texttt{u}_{1}^\top A \texttt{u}_{1}&	\texttt{u}_{1}^\top A \texttt{u}_{2}&\cdots&	\texttt{u}_{1}^\top A \texttt{u}_{\ell} \\[0.2cm]
			\texttt{u}_{2}^\top A \texttt{u}_{1}&	\texttt{u}_{2}^\top A \texttt{u}_{2}&\cdots&	\texttt{u}_{2}^\top A \texttt{u}_{\ell}\\
			\vdots&\vdots&\cdots&\vdots\\
			\texttt{u}_\ell^\top A \texttt{u}_{1}&	\texttt{u}_\ell^\top A \texttt{u}_{2}&\cdots&	\texttt{u}_\ell^\top A \texttt{u}_\ell\\
		\end{array}\right] \\[0.3em]
		&= U_\ell^\top A U_\ell.
	\end{split}
\end{align}

\subsubsection{Finding the eigenpairs of $\mathcal R$}
Next, we assume that $\psi_k$ is the $k$-th orthonormal eigenvector of $K_\ell$ corresponding to the eigenvalue of $ \lambda_k(K_\ell)$, i.e.,  
\begin{align}\label{K_matrix_ei}
	K_\ell \psi_k  = \lambda_k(K_\ell) \psi_k, \quad \textup{and}\quad \psi_k^\top \psi_k=1.
\end{align}

According to \Cref{Eig}, the  nonzero eigenvalues of $K_\ell$ and $\mathcal R$ are the same.
\begin{lemma}\label{R_and_KEll}
	Assume that $\lambda_1(K_\ell)\ge \lambda_2(K_\ell)\ge\ldots\ge \lambda_\ell(K_\ell)>0$ are the eigenvalues of $K_\ell$, then
	\begin{align*}
		\lambda_i(K_\ell) = \lambda_i(\mathcal R), \quad \textup{for}\; i=1,2,\ldots, \ell.
	\end{align*}
\end{lemma}

Then, by \Cref{Eig} we can compute the basis $\widetilde \varphi_k $, $k=1,2,\ldots, \ell$, 
\begin{align*}
	\widetilde \varphi_k = \frac{1}{\sqrt{\lambda_k(K_\ell)}} \mathcal U \psi_k 
	= \frac{1}{\sqrt{\lambda_k(K_\ell)}} \sum_{i=1}^\ell  \left((\psi_{k})_i \sum_{j=1}^{N}(\texttt{u}_{i})_j \varphi_j \right)
	= \frac{1}{\sqrt{\lambda_k(K_\ell)}} \sum_{j=1}^{N} (U_\ell\psi_k)_j \varphi_j.
\end{align*}
In other words, the basis function $\widetilde \varphi_k$ is a finite element function, and  its coefficients are 
$$
\frac{1}{\sqrt{\lambda_k(K_\ell)}}  U_\ell \psi_k.
$$

In practice, truncation is performed when the eigenvalue is small. Assume the first $r$ eigenvalues satisfy our requirement, we then define a matrix $Q$ by 
\begin{align}\label{D_matrix}
	Q = U_\ell [\psi_1\mid \psi_2\mid \ldots| \psi_{r}]\begin{bmatrix}	\frac{1}{\sqrt{\lambda_1(K_\ell)}}  & & \\ & \ddots & \\ & & 	\frac{1}{\sqrt{\lambda_r(K_\ell)}} \end{bmatrix} \in \mathbb R^{N\times r}.
\end{align}

It is obvious that the $i$-th column of $Q$ is the coefficient of the finite element function $\widetilde \varphi_i$, $i=1,2,\ldots, r$. Next, we assume that $\Psi$ and $\Lambda$ are the eigenpair of $K_\ell$, i.e., $K_\ell \Psi = \Psi \Lambda$. For notational convenience, we adopt Matlab notation herein.  We use $\Psi(:, 1:r)$ to denote the first $r$ columns of $\Psi$, and $\Lambda(1:r,1:r)$ be the $r-$th leading principal minor of $\Lambda$. We summarize the above discussion in \Cref{Brrief} (see lines 1 -- 9).

\subsection{Discussion of matrix $K_\ell$ and its eigenvalues}
The main computational cost  in \Cref{Brrief} is to solve $\ell$ linear systems. However, it is unclear how to choose $\ell$. Hence we  have to choose $\ell$  large, this makes the computation expensive. In this section, we remedy the  \Cref{Brrief} such that we do not have any unnecessary cost.

Recall that $M$ and $A$ are the mass and  stiffness matrix; see the definition in \eqref{MassStiffnessload}. Define  $\widehat A = M^{-1}A$ and $\widehat b = M^{-1}b$, and   by \eqref{Coefficients_Matrix} we have 
\begin{align}\label{U_ell}
	U_\ell= [\widehat A^{-1} \widehat b\mid \widehat A^{-2} \widehat b \mid \ldots\mid \widehat A^{-\ell} \widehat b].
\end{align}
The matrix $U_\ell$ is nothing but the so-called Krylov matrix. Let $r\le \ell$ be the rank of $U_\ell$, then
\begin{align*}
	\textup{span}\{\widehat A^{-1} \widehat b,  \widehat A^{-2} \widehat b,  \ldots, \widehat A^{-\ell} \widehat b\} = \textup{span}\{\widehat A^{-1} \widehat b, \widehat A^{-2} \widehat b, \ldots, \widehat A^{-r} \widehat b\}.
\end{align*}

In other words, there is no need to compute $\widehat A^{-(r+1)} \widehat b$, $\ldots$, $\widehat A^{-\ell}\widehat b$, i.e., we only need to solve $r$ linear systems in \eqref{Generate_basis}. By the definition of $K_\ell$ in \eqref{K_matrix}, we see that the matrix $K_r$ is positive definite and $K_{r+1}$ is  positive semi-definite. Hence, we only need to compute the minimal eigenvalue of the matrices $K_1, K_2,\ldots$. Once the minimal eigenvalue of some matrix is zero, we stop. 

In practice, we terminate the process if the minimal eigenvalue  is small. 

Next, we show that the matrix $K_r$ is  Hankel type matrix, i.e.,  each ascending skew-diagonal from left to right is constant. 
\begin{lemma}
	For any $1\le i\le \ell$, $1\le j\le \ell$, we have
	\begin{align*}
		K_{1, i+j-1} = K_{2, i+j-2} = \ldots = K_{i, j} = K_{i+1, j-1} = \ldots =  K_{i+j-1, 1}.
	\end{align*}
	Hence, $K$ is Hankel type matrix.
\end{lemma}
\begin{proof}
	First, we define $\mathcal A_h: V_h \to V_h$ by
	\begin{align}\label{bilinear_Ah1}
		a(w_h, v_h) = (\mathcal A_h w_h, v_h)\quad \text { for all } v_{h} \in V_{h}.	
	\end{align}
	It is easy  to see that  $\mathcal A_h^{-1}$ exists and $\mathcal A_h$ is self-adjoint since the bilinear form $a(\cdot,\cdot)$ is symmetric. By \eqref{Generate_basis} we have
	\begin{align*}
		\mathfrak{u}_h^i = \mathcal A_h^{-1} \mathfrak{u}_h^{i-1} = \mathcal A_h^{-2} \mathfrak{u}_h^{i-2}  = \ldots = \mathcal A_h^{-i} \mathfrak{u}_h^{0}. 
	\end{align*}
	By the definition $K_{\ell}$ in \eqref{K_matrix} we have 
	\begin{align*}
		K_{ij} = a( \mathfrak{u}_h^j, \mathfrak{u}_h^i) = (\mathcal A_h \mathfrak{u}_h^j, \mathfrak{u}_h^i)= (\mathcal A_h \mathfrak{u}_h^j, \mathcal A_h ^{-1}\mathfrak{u}_h^{i-1})= (\mathfrak{u}_h^j, \mathfrak{u}_h^{i-1}) = (\mathcal A_h \mathfrak{u}_h^{j+1}, \mathfrak{u}_h^{i-1}) = K_{i-1, j+1}.
	\end{align*}
	By the same arguments we have
	\begin{align*}
		K_{1, i+j-1} = K_{2, i+j-2} = \ldots = K_{i, j} = K_{i+1, j-1} = \ldots =  K_{i+j-1, 1}.
	\end{align*}
	
\end{proof}

Therefore, to assemble the matrix $K_r$, we only need the matrix $K_{r-1}$ and to compute $\texttt{u}_{r-1}^\top A \texttt{u}_{r}$ and  $ \texttt{u}_{r}^\top A \texttt{u}_{r}$. 
Now we summarize the above discussion in \Cref{D_matrixalgorthm2}.
\begin{algorithm}[H]
	\caption{ (Get the matrix $Q$)}
	\label{D_matrixalgorthm2}
	{\bf{Input}:}  tol, $\ell$, $M$, $A$, $b$
	\begin{algorithmic}[1]
		
		\State Solve $A\texttt{u}_{1} = b$;
		\State Let $K_1 = \texttt{u}_{1}^\top A\texttt{u}_{1}$;
		\For{$i=2$ to $\ell$}
		
		\State Solve $A\texttt{u}_{i} = M \texttt{u}_{i-1}$;
		\State Get $\alpha = [K_{i-1}(i-1,2:i-1) \mid \texttt{u}_{i-1}^\top A \texttt{u}_{i}]$ and  $\beta = \texttt{u}_{i}^\top A \texttt{u}_{i}$;
		\State $K_i = \left[\begin{array}{cc}
			K_{i-1}&\alpha^\top\\
			\alpha&\beta\\
		\end{array}\right]$;
		\State $[\Psi, \Lambda] = \textup{eig}(K_i)$;
		\If{$\Lambda(i,i)\le \textup{tol}$}
		\State break;
		\EndIf	
		\EndFor
		\State Set $U = [\texttt{u}_{1} \mid \texttt{u}_{2} \mid \ldots\mid \texttt{u}_{i}]$;
		\State Set $Q = U\Psi(:,1:r-1)(\Lambda(1:i-1,1:i-1))^{-1/2}$;
		
		\State {\bf return} $Q$
	\end{algorithmic}
\end{algorithm}

Next, we show that the eigenvalues of the Hankel matrix $K_r$ have exponential decay.

\begin{theorem}\label{UAU}
	Let $\lambda_1 ({K_{r}})\ge \lambda_2({K_{r}})\ge \ldots\ge \lambda_r({K_{r}})> 0$ be the eigenvalues of $K_{r}$, then
	\begin{align}\label{all_eigs}
		\lambda_{2k+1} ({K_{r}}) \le 16\left[\exp \left(\frac{\pi^{2}}{4 \log (8\lfloor r / 2\rfloor / \pi)}\right)\right]^{-2k+2}\lambda_1 ({K_{r}}), \qquad 2k+1\le r.
	\end{align}
	Here $\lfloor\cdot\rfloor$ is the floor function  that takes as input a real number $x$, and gives as output the greatest integer less than or equal to $x$.  
	Moreover, the minimal eigenvalue of $K_r$ satisfies
	\begin{align}\label{last_eigs}
		\lambda_{\min} ({K_{r}})  \le C  (2r-1) \|f\|_{V'}^2 \exp\left(-\dfrac{7(r+1)}{2}\right).
	\end{align}
\end{theorem}

The proof of \Cref{UAU}  and the error analysis in \Cref{Error_analysis} relies on the discrete eigenvalue problem of  the bilinear form $a(\cdot, \cdot)$.  Let $(\lambda_{ih}, \phi_{ih})\in \mathbb R\times V_h$ be the solution of 
\begin{align}\label{eigfunctions}
	a \left(\phi_{ih},  v_h\right)=\lambda_{ih}\left(\phi_{ih}, v_h\right) \quad \text { for all } v_h \in V_h.	
\end{align}

It is well known that the eigenvalue problem \eqref{eigfunctions} has a finite sequence of eigenvalues and  eigenfunctions
\begin{gather*}
	0<\lambda_{1,h} \leq \lambda_{2,h} \leq \ldots\le \lambda_{N,h}, \quad  \phi_{1,h},	\phi_{2,h},  \ldots,	\phi_{N,h},\quad (\phi_{i,h}, \phi_{j,h})_V=\delta_{ij}.
\end{gather*}

Let $\widehat \phi_i\in \mathbb R^N$ be the coefficient of $\phi_{i,h}$ in terms of  the finite element basis $\{\varphi_j\}_{j=1}^N$. Since $(\phi_{i,h}, \phi_{j,h})_V=\delta_{ij}$, then $
\widehat \phi_i^\top A\widehat \phi_j= \delta_{ij}$. 
This implies that $\{\widehat \phi_i\}_{i=1}^N$ is a $A$-orthnormal basis in $\mathbb R^N$. 

Let $\Pi: V\to V_h$ be the standard $L^2$ projection, i.e., for all $w\in L^2(\Omega)$ we have 
\begin{align}\label{L2Projection}
	(\Pi w, v_h) = (w, v_h),\quad \forall v_h\in V_h.
\end{align}
\begin{lemma}\label{coeff}
	Let  $\{\mathfrak{b}_j\}_{j=1}^N$ and $\{\mathfrak{f}_j\}_{j=1}^N$ be the coefficient of $\widehat b$ and $\Pi f$ in terms of  $\{\widehat \phi_i\}_{i=1}^N$ and $\{ \phi_{i,h}\}_{i=1}^N$, respectively. Then
	\begin{align*}
		\mathfrak{b}_j = \mathfrak{f}_j,\qquad j=1,2\ldots, N.
	\end{align*}
	Furthermore, 
	\begin{align*}
		\sum_{j=1}^N  \frac{\mathfrak b_j^2}{\lambda_{j,h}^2}  \le C \|f\|_{V'}^2.
	\end{align*}
\end{lemma}
\begin{proof}
	First, we take $w_h = \mathcal A_h^{-1} \Pi f$ in \eqref{bilinear_Ah1} to obtain:
	\begin{align*}
		a(\mathcal A_h^{-1} \Pi f, v_h) = (\Pi f, v_h) = (f, v_h).
	\end{align*}
	By the continuity of the bilinear form $a(\cdot, \cdot)$ we have 
	\begin{align*}
		\|\mathcal A_h^{-1} \Pi f\|_V \le C \|f\|_{V'}.
	\end{align*}
	Since $\widehat b$ is the coefficient of $\Pi f$ in terms of the finite element basis $\{\varphi_i\}_{i=1}^N$, then 
	\begin{align*}
		\mathfrak{f}_j = (\Pi f, \phi_{j,h})_V = \left(\sum_{i=1}^N (\widehat b)_{\textcolor{blue}{i}}\varphi_i, \sum_{k=1}^N (\widehat \phi_j)_k \varphi_k\right)_V = \widehat \phi_j^\top A\widehat b = \widehat \phi_j^\top A \sum_{\textcolor{blue}{i}=1}^N \mathfrak{b}_{\textcolor{blue}{i}}\widehat \phi_{\textcolor{blue}{i}} = \mathfrak{b}_j.
	\end{align*}
	This implies
	\begin{align*}
		\Pi f = \sum_{j=1}^N  \mathfrak{b}_j\phi_{j,h}.
	\end{align*}
	Note that $(\phi_{i,h}, \phi_{j,h})_V = \delta_{ij}$, therefore we have 
	\begin{align*}
		\|\mathcal A_h^{-1} \Pi f\|_V = \left\|\mathcal A_h^{-1} \sum_{j=1}^N 	\mathfrak{b}_j \phi_{j,h}\right\|_V = \left\| \sum_{j=1}^N 	\frac{\mathfrak{b}_j}{\lambda_{j,h}} \phi_{j,h}\right\|_V = \left(\sum_{j=1}^N \frac{\mathfrak{b}_j^2}{\lambda_{j,h}^2}\right)^{1/2}.
	\end{align*}
	This implies, for any $N$, we have 
	\begin{align*}
		\sum_{j=1}^N \frac{\mathfrak{b}_j^2}{\lambda_{j,h}^2} \le C \|f\|_{V'}^2.
	\end{align*}
\end{proof}

Let $\mu_j = 1/\lambda_{j,h}$, we define the rectangular Vandermonde matrix $V_d$ and weighted matrix $W$ by
\begin{align*}
	V_d = 
	\left[\begin{array}{cccc}
		1&1&\cdots&1\\[0.2cm]
		\mu_1&\mu_2&\cdots&\mu_N\\
		\vdots&\vdots&\cdots&\vdots\\
		\mu_1^{r-1}&\mu_2^{r-1}&\cdots&\mu_N^{r-1}\\
	\end{array}\right], \qquad 		W = \left[\begin{array}{cccc}
		\mathfrak{b}_1\mu_1&0&\cdots&0\\[0.2cm]
		0&\mathfrak{b}_2\mu_2&\cdots&0\\
		\vdots&\vdots&\cdots&\vdots\\
		0&0&\cdots&\mathfrak{b}_N\mu_N\\
	\end{array}\right]. 
\end{align*}

\begin{lemma}
	The Hankel matrix $K_r$, which was defined in \eqref{K_matrix} can be rewritten as
	\begin{align*}
		K_r=  (V_dW) (V_dW)^\top.
	\end{align*}
\end{lemma}
\begin{proof}
	Recall \eqref{eigfunctions} and note that $\widehat \phi_i\in \mathbb R^N$ is the coefficient of $\phi_{i,h}$ in terms of  the finite element basis $\{\varphi_j\}_{j=1}^N$, then
	\begin{align*}
		A \widehat \phi_i = \lambda_{i,h} M \widehat \phi_i. 
	\end{align*}
	This gives
	\begin{align}\label{RN_eig}
		\widehat A \widehat \phi_i  = \lambda_{i,h}  \widehat \phi_i.
	\end{align}
	Then, by \eqref{U_ell} and \eqref{RN_eig} we have 
	\begin{align*}
		U_r&= [\widehat A^{-1} \widehat b\mid \widehat A^{-2} \widehat b \mid \ldots\mid \widehat A^{-r} \widehat b]\\
		& = \left[\sum_{j=1}^N \mathfrak{b}_j\mu_j\widehat \phi_j ~\bigg|~ \sum_{j=1}^N {\mathfrak{b}_j}\mu_j^2\widehat \phi_j  ~\bigg|~ \ldots\ldots  ~\bigg|~  \sum_{j=1}^N {\mathfrak{b}_j}\mu_j^r\widehat \phi_j\right]\\
		& =[\widehat \phi_1\mid \widehat \phi_2 \mid \ldots\mid \widehat \phi_N] WV_d^\top.
	\end{align*}
	Therefore, by  \eqref{K_matrix}  we have 
	\begin{align*}
		K_r= U_r^\top A U_r  = (WV_d^\top)^\top \underbrace{[\widehat \phi_1\mid \widehat \phi_2 \mid \ldots\mid \widehat \phi_N]^\top A [\widehat \phi_1\mid \widehat \phi_2 \mid \ldots\mid \widehat \phi_N]}_{I} WV_d = (V_dW) (V_dW)^\top.
	\end{align*}
\end{proof}

\begin{proof}[Proof of \Cref{UAU}]
	First, the estimate \eqref{all_eigs} holds since $K_r$ is a positive definite Hankel matrix; see \cite[Corollary 5.5]{MR3947283}. Next, we prove \eqref{last_eigs} and we use the techniques in \cite{MR1273642,MR909407}.  
	
	Without loss of generality, we assume that the minimal  eigenvalue of \eqref{eigfunctions} $\lambda_{1,h}$ is  no less   than $1$; otherwise we modify the eigenvalue problem \eqref{eigfunctions} by shifting.
	
	Let $\mathcal L_k(x) = \displaystyle\sum_{j=1}^k \ell_{kj} x^{j-1}$ be the $k$-th Legendre polynomial in $[0,1]$, such that
	\begin{align*}
		\int_0^{1} \mathcal L_i(x)\mathcal L_j(x) {\rm d}x = \delta_{ij} \quad \textup{and}\quad \max_{x\in [0,1]} |\mathcal L_k(x)| = \sqrt{{2k-1}}.
	\end{align*}
	We define the vector $v_k= [\ell_{k1},\cdots \ell_{kk},0,\cdots,0 ]^\top \in \mathbb{R}^r$ with $1\le k\le r$ and matrix $L$ by
	\begin{align*}
		L = \left[\begin{array}{cccc}
			\ell_{11}&0&0&0\\[0.2cm]
			\ell_{11}&\ell_{12}&\cdots&0\\
			\vdots&\vdots&\cdots&\vdots\\
			\ell_{r1}&\ell_{r2}&\cdots&\ell_{rr}\\
		\end{array}\right].
	\end{align*}
	We note that $L$ is non-singular; all entries from the main diagonal of $L$ are different from zero since the degree of $\mathcal L_k$ is exactly $k-1$. By the definition of Legendre polynomials we have 
	\begin{align*}
		\left[\begin{array}{cc}
			\mathcal L_1 \\[0.2cm]
			\mathcal L_2\\
			\vdots\\
			\mathcal L_r
		\end{array}\right]=	 \left[\begin{array}{cccc}
			\ell_{11}&0&0&0\\[0.2cm]
			\ell_{11}&\ell_{12}&\cdots&0\\
			\vdots&\vdots&\cdots&\vdots\\
			\ell_{r1}&\ell_{r2}&\cdots&\ell_{rr}\\
		\end{array}\right]\left[\begin{array}{cc}
			1  \\[0.2cm]
			x\\
			\vdots\\
			x^{r-1}
		\end{array}\right]. 
	\end{align*}
	Therefore,
	\begin{align*}
		\left[\begin{array}{cc}
			1  \\[0.2cm]
			x\\
			\vdots\\
			x^{r-1}
		\end{array}\right] =  \left[\begin{array}{cccc}
			\ell_{11}&0&0&0\\[0.2cm]
			\ell_{11}&\ell_{12}&\cdots&0\\
			\vdots&\vdots&\cdots&\vdots\\
			\ell_{r1}&\ell_{r2}&\cdots&\ell_{rr}\\
		\end{array}\right]^{-1} \left[\begin{array}{cc}
			\mathcal L_1 \\[0.2cm]
			\mathcal L_2\\
			\vdots\\
			\mathcal L_r
		\end{array}\right].
	\end{align*}
	Let $D$ be the inverse of $L$, then for $i=1,2,\ldots, r$ we have
	\begin{align}\label{inver}
		x^{i-1} = \sum_{j=1}^r d_{ij} \mathcal L_j(x).
	\end{align} 
	We multiply	$\mathcal L_j(x)$ on both sides of \eqref{inver}  and integrate on $[0, 1]$ to obtain:
	\begin{align}\label{dmatrix}
		\int_0^{1} x^{i-1} \mathcal L_j(x) {\rm d}x = d_{ij}.
	\end{align}	
	By virtue of orthogonality of Legendre polynomials, and  \eqref{inver} and \eqref{dmatrix}  imply 
	\begin{align*}
		\sum_{j=1}^k\ell_{kj} \int_0^1 x^{i-1} x^{j-1} {\rm d}x = d_{ij}.
	\end{align*}
	In other words,
	\begin{align}\label{Hil}
		H_r L^\top = L^{-1},
	\end{align}
	where $H_r$ is the $r\times r$ Hilbert matrix with $(H_r)_{ij}= 1/(i+j-1)$. It is well known that the condition number of $H_r$ grows like $(1+\sqrt 2)^{4r}/\sqrt{r}$; see \cite[Eq.(3.35)]{MR0271762}. Next, by~\eqref{Hil} we know
	\begin{align*}
		H_r^{-1} = \sum_{k=1}^r v_k v_k^\top.
	\end{align*}
	Now, let $B=\displaystyle\sum_{k=1}^{r-1} v_kv_k^\top$, and one can verify that $B$ has the block form $\begin{pmatrix}
		H_{r-1}^{-1} & 0 \\ 0 & 0
	\end{pmatrix}$. Thus $B$ has a zero eigenvalue and its other (positive) eigenvalues are the reciprocals of the eigenvalues of $H_{r-1}$. Moreover, since $H_r^{-1}-B= v_r v_r^\top$ is a rank one matrix, we have 
	\begin{align*}
		\lambda_{\max} (H_r^{-1}-B) = v_r^\top v_r.
	\end{align*}
	On the other hand, 
	\begin{align*}
		v_r^\top v_r &=	\lambda_{\max} (H_r^{-1}-B) \ge \lambda_{\max}(H_r^{-1}) - \lambda_{\max}(B)
		= \frac{1}{\lambda_{\min} (H_r)} - \frac{1}{\lambda_{\min} (H_{r-1})}
		= \frac{1-c_r}{\lambda_{\min} (H_r)},
	\end{align*}
	where $c_r= \frac{\lambda_{\min} (H_r)}{\lambda_{\min} (H_{r-1})}$. We remark that for $r\ge 2$ one has $(1-c_r)^{-1}<6/5$. Note that
	\begin{align*}
		(V_dW)^\top v_r = \left[\begin{array}{cccc}
			\mathfrak b_1\mu_1 &\mathfrak b_1\mu_1^2&\cdots&\mathfrak b_1\mu_1^{r}\\[0.2cm]
			\mathfrak b_2\mu_2& \mathfrak b_2\mu_2^2&\cdots& \mathfrak b_2\mu_2^{r}\\
			\vdots&\vdots&\cdots&\vdots\\
			\mathfrak b_N \mu_N& \mathfrak b_N\mu_N^2 &\cdots&\mathfrak b_N\mu_N^{r}\\
		\end{array}\right] \left[\begin{array}{cccc}
			\ell_{r1}\\
			\ell_{r2}\\
			\vdots\\
			\ell_{rr}\\
		\end{array}\right] = \left[\begin{array}{cccc}
			\mathfrak b_1\mu_1\mathcal L_r(\mu_1)\\
			\mathfrak b_2\mu_2\mathcal L_r(\mu_2)\\
			\vdots\\
			\mathfrak b_N\mu_N\mathcal L_r(\mu_N)\\
		\end{array}\right].
	\end{align*} 
	Therefore 
	\begin{align*}
		\lambda_{\min}(K_r) & \le \frac{v_r^\top K_r v_r}{v_r^\top v_r} = \frac{v_r^\top (V_dW) (V_dW)^\top v_r}{v_r^\top v_r}=  \frac{ ((V_dW)^\top v_r)^\top (V_dW)^\top v_r}{v_r^\top v_r}\\
		&= \dfrac{\displaystyle\sum_{k=1}^N \mathfrak b_k^2 \mu_k^2 \mathcal L_r^2(\mu_k)}{v_r^\top v_r}  \le \max_{x\in [0,1]}|\mathcal L_r(x)|^2 \frac{\lambda_{\min}(H_r) }{1-c_r} \sum_{k=1}^N \mathfrak b_k^2\mu_k^2\\
		&\le C(2r-1)\lambda_{\min}(H_r)\|f\|_{V'}^2.
	\end{align*}
	Recall \cite[Eq.(3.35)]{MR0271762} that $	\lambda_{\min} (H_r)\le C \exp(-7(r+1)/2)$. Then
	\begin{align*}
		\lambda_{\min}(K_r) \le C (2r-1) \|f\|_{V'}^2 \exp(-7(r+1)/2).
	\end{align*}
\end{proof}

\begin{example}\label{Example2}
	We use the same problem data as in the \Cref{Example0} and take $h=1/100$. We report  all the eigenvalues of $K_8$ in \Cref{2DEigs}. It is clear that the eigenvalues have exponential decay. This matches our theoretical result in \Cref{UAU}.
	\begin{figure}[tbh]
		\centerline{
			\hbox{\includegraphics[width=3.5in]{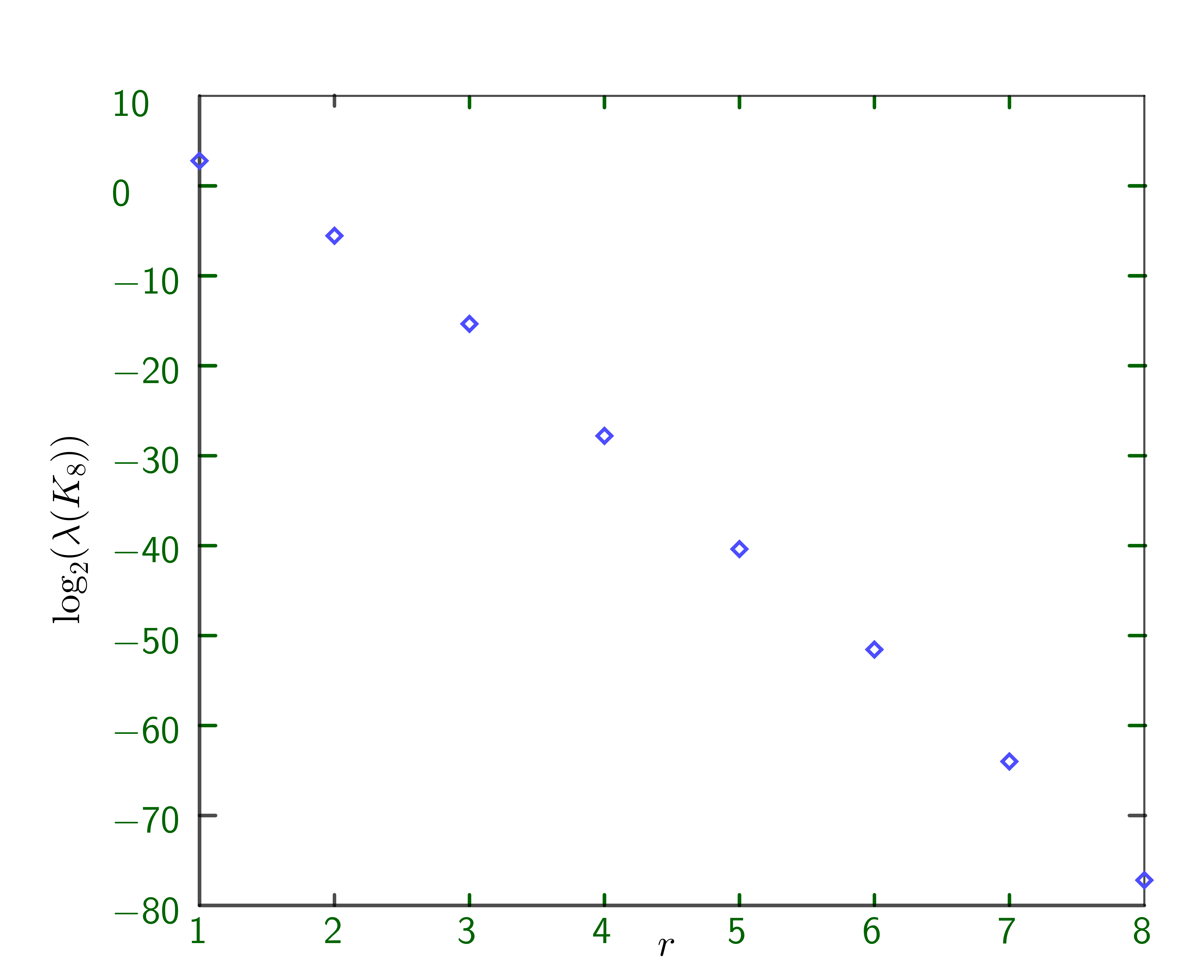}}}
		\caption{The eigenvalues of the Hankel matrix $K_8$.}
		\label{2DEigs}
		\centering
	\end{figure}
	
\end{example}

\subsection{Reduced order model (ROM)}
In the rest of the paper, we let $\widetilde \varphi_1, \widetilde\varphi_2,\ldots, \widetilde\varphi_{r}$ be the order $r$ optimal reduced basis, which was  generated in \Cref{D_matrixalgorthm2},  and we define $V_r$  by
\begin{align}\label{PODbasis0}
	V_r = \textup{span}\{\widetilde \varphi_1, \widetilde\varphi_2,\ldots, \widetilde\varphi_{r}\}.
\end{align}

Using the space $V_r$ we will construct the following  two-steps backward differentiation reduced-order model (BDF2-ROM) scheme. Find $u_r^n\in V_r$ with $u_r^0=0$ satisfying
\begin{align}\label{ROM}
	\left(\partial_t^+ u_r^n, v_r\right)+ a\left(u_r^n, v_r\right) = (f, v_r), \quad \forall v_r\in V_r.
\end{align}
Now, we are ready to compute the inner product in \eqref{ROM}. Let $q_j$ be the $j$-th column of $Q$, then
\begin{align*}
	A_{r, ij} = a(\widetilde \varphi_j, \widetilde \varphi_i) =  \left(\sum_{m=1}^{N} (q_{j})_m\varphi_m,  \sum_{n=1}^{N} (q_{i})_n\varphi_n\right)_V = q_j^\top A q_i,
\end{align*}
where $A$ is the stiffness matrix defined in \eqref{MassStiffnessload}. Hence,
\begin{align*}
	A_r = \left[\begin{array}{cccc}
		q_1^\top A q_1&	q_1^\top A q_2&\cdots&	q_1^\top A q_{r} \\[0.2cm]
		q_2^\top A q_1&	q_2^\top A q_2&\cdots&	q_2^\top A q_{r}\\
		\vdots&\vdots&\cdots&\vdots\\
		q_{r}^\top A q_1&	q_{r}^\top A q_2&\cdots&	q_{r}^\top A q_{r}\\
	\end{array}\right] = Q^\top A Q\in \mathbb R^{r\times r}.
\end{align*}
Similarly, we define the reduced mass matrix and load vector.
\begin{align}\label{reduced_matrices}
	M_r = Q^\top M Q\in  \mathbb R^{r\times r}, \quad  \quad b_r =  Q^\top   b \in \mathbb R^{r}.
\end{align}

Finally, we are ready  to complete the full implementation of \eqref{ROM}.
Since $u_r^n\in V_{r}$ holds, we make the Galerkin ansatz of the form
\begin{align}\label{Reduced_Galerkin0}
	u_r^n=\sum_{j=1}^{r} \left(\alpha_r^n\right)_j {\widetilde \varphi}_{j}\quad \text { for } n=1,2\ldots, N_T.
\end{align}
We insert \eqref{Reduced_Galerkin0} into \eqref{ROM}  and have the following linear systems :
\begin{align}\label{semidiscrete1}
	\begin{split}
		M_r\partial_t^+ \alpha_r^n + A_r  \alpha_r^n  &=  b_{r},\quad n\ge 1,\\
		\alpha_r^0 &\textit{}= 0.
	\end{split}
\end{align}

For convenience, we should express the solution $u_r^n$ by using the standard finite element basis. By \eqref{D_matrix} and \eqref{Reduced_Galerkin0} we have 
\begin{align*}
	u_r^n=\sum_{j=1}^{r} \left(\alpha_r^n\right)_j {\widetilde \varphi}_{j}  = \sum_{j=1}^{r} \left(\alpha_r^n\right)_j \sum_{i=1}^{N} (q_j)_i \varphi_i = \sum_{i=1}^{N} (Q\alpha_r^n)_i\varphi_i.
\end{align*}

In other words, the coefficients of the solution $u_r^n$ in terms of the standard finite element basis are $Q\alpha_n$. It is worthwhile mentioning that to store the solution, we only need to save the matrix $Q\in \mathbb R^{N\times r}$ and the coefficient $\alpha_n\in \mathbb R^{r}$. Next, we summarize the above discussion in \Cref{algorithm3}.
\begin{algorithm}
	\caption{ (Fully implementation of \eqref{ROM})}
	\label{algorithm3}
	{\bf{Input}:}  \texttt{tol}, $\ell$, $N_T$, $\Delta t$, $M$, $A$, $b$
	\begin{algorithmic}[1]
		\State  $Q$ = \texttt{GetTheMatrixQ}$(\texttt{tol},\ell,M,A,b)$;  \hspace{6.5 cm} \% \Cref{D_matrixalgorthm2}
		\State Set $M_r = Q^\top M Q$; $A_r = Q^\top A Q$; $b_r = Q^\top b$;
		\State Solve $\left(\dfrac{1}{\Delta t}M_{r}+ A_r\right) \alpha_r^1= b_{r}$;
		\For {$n = 2$ to $N_T$}	
		\State Solve $\left(\dfrac{3}{2\Delta t}M_{r}+A_r\right) \alpha_r^n=  \dfrac{1}{\Delta t}M_{r}\left(2\alpha_r^{n-1} - \dfrac 1 2 \alpha_r^{n-2}\right)+   b_{r}$;
		\EndFor
		\State {\bf return} $Q, \alpha_n$
	\end{algorithmic}
\end{algorithm}

\begin{example}\label{Example3}
	We revisit the \Cref{Example0} with the same problem data, mesh and time step.   We choose $\ell=10$, \texttt{tol} $= 10^{-14}$ in \Cref{algorithm3}.  We report the dimension $r$ and the wall  time of the ROM  in \Cref{table_1}. Comparing with \Cref{table_0}, we see that our ROM is much faster than  standard solvers. We also compute the $L^2$-norm error between the solutions of the FEM and the ROM at the final time $T=1$, we see that the error is close to the machine error. This motivated us that the solutions of the FEM and of the ROM are the same if we take \texttt{tol} small enough. In \Cref{Error_analysis} we give a rigorous error analysis under an assumption on the source term $f$.
	\begin{table}[H]
		\centering
		{
			\begin{tabular}{c|c|c|c|c|c|c|c}
				\Xhline{1pt}
				
				\cline{3-8}
				$h$	& $1/2^4$ &$1/2^5$ 
				&$1/2^6$ 
				&$1/2^7$  &$1/2^8$ &$1/2^9$  &$1/2^{10}$ 
				\\
				\cline{1-8}
				{$r$}
				& 	  6&   6&   6&   6&   6&   6&   6\\ 
				\cline{1-8}
				{Wall time (s)}
				& 	  0.03&   0.04&   0.13&   0.62&   2.03&   9.86&   44.3\\ 
				\cline{1-8}
				{Error}
				&6.73E-11& 4.78E-15&	  2.03E-15&   7.50E-14&   6.87E-13&  7.01E-13&   7.88E-12\\ 
				\Xhline{1pt}

			\end{tabular}
		}
		\caption{\Cref{Example3}: The dimension and wall time (seconds) of the  ROM. The $L^2$-norm error between the solutions of the FEM and the ROM at the final time.}\label{table_1}
	\end{table}		
\end{example}

\section{Error Analysis}\label{Error_analysis}
Next, we provide a fully-discrete convergence analysis of the above new reduced order method for  linear parabolic equations.	Throughout the section, the  constant $C$  depends on the polynomial degree $k$, the domain, the shape regularity of the mesh and the problem data. But, it does not depend on the mesh size $h$, the time step $\Delta t$ and the dimension of the ROM.

\subsection{Main assumption and main result} 
First, we recall that $\Pi$ is the standard $L^2$ projection (see \eqref{L2Projection}) and $\{\phi_{j,h}\}_{j=1}^N$ are the eigenfuctions of \eqref{eigfunctions} corresponding to the eigenvalues $\{\lambda_{j,h}\}_{j=1}^N$.

Next, we give our main assumptions in this section:
\begin{assumption}\label{Assumption1}
	Low rank of $f$:  there exist  $\{c_j\}_{j=1}^\ell$ such that
	\begin{align}\label{Ass1}
		\Pi f = \sum_{j=1}^\ell c_j \phi_{m_jh}.
	\end{align}
\end{assumption}
\begin{assumption}\label{Assumption2}
	The regularity of the solution of \eqref{Parabolic_Original_Model} is smooth enough.
\end{assumption}
Now,  we state our main result in this section:
\begin{tcolorbox}
	\begin{theorem}\label{Main_res}
		Let $u$ be the solution of \eqref{Parabolic_Original_Model} and $u_r^n$  be the solution of \eqref{ROM} by setting \textup{\texttt{tol}}$ = 0$ in \Cref{algorithm3}.  If Assumption \ref{Assumption1} and Assumption \ref{Assumption2} hold, then we have 
		\begin{align*}
			\left\|u(t_n) - u_r^n\right\|  \le C\left( h^{k+1} + (\Delta t)^2 \right).
		\end{align*}
	\end{theorem}			
\end{tcolorbox}

\subsection{Sketch the proof of \Cref{Main_res}}
To prove \Cref{Main_res}, we first bound the error
between the solutions of the PDE  \eqref{Parabolic_Original_Model} and FEM
\eqref{discreteODE}. Next we prove that  the solutions of FEM
\eqref{discreteODE} and the ROM \eqref{ROM} are exactly the same. Then we obtain  a bound on the error between the solutions of PDE \eqref{Parabolic_Original_Model} and the ROM \eqref{ROM}.

We begin by bounding the error between the solutions of the FEM \eqref{discreteODE} and PDE \eqref{Parabolic_Original_Model}. 
\begin{lemma}\label{lemma1}
	Let $u$ and $ u_h^n$ be the solution of \eqref{Parabolic_Original_Model} and \eqref{discreteODE}, respectively. If Assumption \ref{Assumption2} holds, then we have
	\begin{align*}
		\left\|u(t_n) - u_h^n\right\|  \le C\left( h^{k+1} + (\Delta t)^2 \right).
	\end{align*}
\end{lemma}

The proof of \Cref{lemma1} follows from standard estimates of  finite element methods, therefore we skip it. Next, we prove that  the solutions of FEM
\eqref{discreteODE} and the ROM \eqref{ROM} are exactly the same. 
\begin{lemma}\label{lemma2}
	Let $ u_h^n$ be the solution of \eqref{discreteODE} and $u_r^n$ be the solution of \eqref{ROM}  by setting \textup{\texttt{tol}}$ = 0$ in \Cref{algorithm3}.  If \Cref{Assumption1} holds, then for all $n=1,2,\ldots, N_T$ we have 
	\begin{align*}
		u_h^n  =  u_r^n.
	\end{align*}
\end{lemma}

As a consequence, \Cref{lemma1,lemma2}  give the proof of \Cref{Main_res}. 

\subsection{Proof of \Cref{lemma2}} 
Since the eigenvalue problem \eqref{eigfunctions} might have repeated eigenvalues, without loss of generality, we assume that only $\phi_{m_1,h}$ and $\phi_{m_2,h}$ share the same eigenvalues $\lambda_{m_1,h} = \lambda_{m_2,h}$. Recall that $\mu_i = 1/\lambda_{i,h}$, then we have
\begin{align}\label{eig_ass}
	\mu_{m_1} = 	\mu_{m_2}>\mu_{m_3} >\ldots > \mu_{m_\ell}. 
\end{align}

By \eqref{bilinear_Ah1} we know that $\phi_{i,h}$ is the eigenfuntion of  $\mathcal A_h^{-1}$ corresponding to the eigenvalue $\mu_i$. Now, we apply $\mathcal A_h^{-1}$, $\mathcal A_h^{-2}$, $\ldots$, $\mathcal A_h^{-\ell}$ to  both sides of \eqref{Ass1} to obtain:
\begin{align}\label{Vandemonde}
	\left[\begin{array}{cc}
		\mathfrak{u}_h^1  \\[0.2cm]
		\mathfrak{u}_h^2\\
		\vdots\\
		\mathfrak{u}_h^\ell
	\end{array}\right] = \left[\begin{array}{cccc}
		c_1 \mu_{m_1}&c_2 \mu_{m_2}&\cdots&c_\ell\mu_{m_\ell}\\[0.2cm]
		c_1 \mu_{m_1}^2&c_2 \mu_{m_2}^2&\cdots&c_\ell\mu_{m_\ell}^2\\
		\vdots&\vdots&\cdots&\vdots\\
		c_1 \mu_{m_1}^{\ell}&c_2 \mu_{m_2}^\ell&\cdots&c_\ell\mu_{m_\ell}^\ell
	\end{array}\right] \left[\begin{array}{cc}
		\phi_{m_1,h}  \\[0.2cm]
		\phi_{m_2,h}\\
		\vdots\\
		\phi_{m_\ell,h}
	\end{array}\right].
\end{align}
By the assumption \eqref{eig_ass}, the rank of the coefficient matrix in \eqref{Vandemonde} is $\ell -1$. Hence 
\begin{align*}
	\textup{span}\{	\mathfrak{u}_h^1, 	\mathfrak{u}_h^2,\ldots, 	\mathfrak{u}_h^{\ell-1}\} = 	\textup{span}\{	\mathfrak{u}_h^1, 	\mathfrak{u}_h^2,\ldots, 	\mathfrak{u}_h^{\ell}\}. 
\end{align*}
This implies that the matrix $K_{\ell-1}$ (see \eqref{K_matrix}) is positive definite and $K_{\ell}$ is positive semi-definite. Therefore, if we set \textup{\texttt{tol}}$ = 0$ in the \Cref{D_matrixalgorthm2}, then the reduced basis space $V_r$ is given by
\begin{align*}
	V_r =  \textup{span}\{\widetilde \varphi_1, \widetilde\varphi_2,\ldots, \widetilde\varphi_{\ell-1}\} =	\textup{span}\{	\mathfrak{u}_h^1, 	\mathfrak{u}_h^2,\ldots, 	\mathfrak{u}_h^{\ell-1}\}.
\end{align*}
Furthermore, by \Cref{Optimization}, for any $j=1,2,\ldots, \ell-1$ we have 
\begin{align*}
	\mathfrak{u}_h^j = \sum_{i=1}^{\ell-1}\left( \mathfrak{u}_h^j, {\widetilde\varphi}_{i}\right)_{V} {\widetilde\varphi}_{i}.
\end{align*}

For $i=1,2\ldots, \ell$, we define  the sequence $\{\alpha_i^n\}_{n=1}^{N_T}$ by
\begin{align}\label{def_alpha}
	\begin{split}
		\partial_t^+\alpha_i^{n}  +  \dfrac 1 {\mu_{m_i}} \alpha_i^{n} &=  c_i, \qquad n\ge 1,\\
		\alpha_i^{0} & = 0.
	\end{split}
\end{align}

\begin{lemma}\label{NonChang}
	If  Assumption \ref{Assumption1} holds, then the unique solution of \eqref{discreteODE} is given by
	\begin{align}\label{Exact_Solution}
		u_h^n  = \sum_{i=1}^\ell \alpha_i^n
		\phi_{m_i,h}, \qquad  n=1,2,\ldots, N_T.
	\end{align} 
\end{lemma}	
\begin{proof}
	We only need  to check that \eqref{Exact_Solution} satisfies \eqref{discreteODE}. Substitute  	\eqref{Exact_Solution} into \eqref{discreteODE} we have 
	\begin{align*}
		\left(\partial_t^+ u_h, v_{h}\right)+  a\left( u_h^n,  v_{h}\right) 
		& = \left( \sum_{i=1}^\ell (\partial_t^+\alpha_i^n)\phi_{m_i,h} , v_{h}\right)+  a\left(\sum_{i=1}^\ell \alpha_i^n\phi_{m_i,h},  v_{h}\right)\\
		& = \left( \sum_{i=1}^\ell (\partial_t^+\alpha_i^n)\phi_{m_i,h} , v_{h}\right)+  \left(\sum_{i=1}^\ell \frac{\alpha_i^n}{\mu_{m_i}}\phi_{m_i,h},  v_{h}\right),
	\end{align*}
	where we used the fact that $\phi_{m_i,h}$ is the eigenvector of \eqref{eigfunctions} corresponding to the eigenvalue $1/\mu_{m_i}$ in the last equality. Therefore, by  \eqref{def_alpha} and Assumption \ref{Assumption1} we have 
	\begin{align*}
		\left(\partial_t^+ u_h, v_{h}\right)+  a\left( u_h^n,  v_{h}\right)  = 	\left(\sum_{i=1}^\ell c_i \phi_{m_i,h}, v_{h}\right) = (\Pi f, v_h) = (f, v_h).
	\end{align*}
\end{proof}

Due to the assumption \eqref{eig_ass}, it is easy to show that for all $n=1,2,\ldots, N_T$,
\begin{align}\label{c1c2}
	\alpha_1^n c_2 = \alpha_2^n c_1.
\end{align}

\begin{lemma}\label{NonChang22}
	Let  $ u_h^n$ be the solution of \eqref{discreteODE} and set \textup{\texttt{tol}}$ = 0$ in \Cref{algorithm3}.  If Assumption \ref{Assumption1} and \eqref{eig_ass} hold, then for $n=1,2,\ldots, N_T$ we have
	\begin{align*}
		u_h^n \in V_r=	\textup{span}\{	\mathfrak{u}_h^1, 	\mathfrak{u}_h^2,\ldots, 	\mathfrak{u}_h^{\ell-1}\}.
	\end{align*}
\end{lemma}

\begin{proof}
	We rewrite the system \eqref{Vandemonde} as
	\begin{align}\label{Vandemonde1}
		\left[\begin{array}{cc}
			\mathfrak{u}_h^1  \\[0.2cm]
			\mathfrak{u}_h^2\\
			\vdots\\
			\mathfrak{u}_h^{\ell-1}
		\end{array}\right] = \left[\begin{array}{cccc}
			c_2 \mu_{m_2}&c_3 \mu_{m_3}&\cdots&c_\ell\mu_{m_\ell}\\[0.2cm]
			c_2 \mu_{m_2}^2&c_3 \mu_{m_3}^2&\cdots&c_\ell\mu_{m_\ell}^2\\
			\vdots&\vdots&\cdots&\vdots\\
			c_2 \mu_{m_2}^{\ell-1}&c_3 \mu_{m_3}^{\ell-1}&\cdots&c_\ell\mu_{m_\ell}^{\ell-1}
		\end{array}\right] \left[\begin{array}{cc}
			\phi_{m_2,h}  \\[0.2cm]
			\phi_{m_3,h}\\
			\vdots\\
			\phi_{m_\ell,h}
		\end{array}\right] + \left[\begin{array}{cc}
			c_1 \mu_{m_1} \\[0.2cm]
			c_1 \mu_{m_1}^2\\
			\vdots\\
			c_1\mu_{m_1,h}^{\ell-1}
		\end{array}\right]	\phi_{m_1,h}. 
	\end{align}	
	
	We denote the coefficient matrix of \eqref{Vandemonde1} by $S$, it is obvious that $S$ is invertible since $\mu_{m_2}$, $\mu_{m_3}$, $\ldots$,  $\mu_{m_\ell}$ are distinct. Furthermore, since $\mu_{m_1} = \mu_{m_2}$, then 
	\begin{align*}
		S\left[\begin{array}{cc}
			1 \\[0.2cm]
			0\\
			\vdots\\
			0
		\end{array}\right] = \left[\begin{array}{cc}
			c_2 \mu_{m_2} \\[0.2cm]
			c_2 \mu_{m_2}^2\\
			\vdots\\
			c_2\mu_{m_2,h}^{\ell-1}
		\end{array}\right]= \left[\begin{array}{cc}
			c_2 \mu_{m_1} \\[0.2cm]
			c_2 \mu_{m_1}^2\\
			\vdots\\
			c_2\mu_{m_1,h}^{\ell-1}
		\end{array}\right].
	\end{align*} 
	This implies
	\begin{align*}
		\left[\begin{array}{cc}
			\phi_{m_2,h}  \\[0.2cm]
			\phi_{m_3,h}\\
			\vdots\\
			\phi_{m_\ell,h}
		\end{array}\right]  = S^{-1}\left[\begin{array}{cc}
			\mathfrak{u}_h^1  \\[0.2cm]
			\mathfrak{u}_h^2\\
			\vdots\\
			\mathfrak{u}_h^{\ell-1}
		\end{array}\right]  - \dfrac{c_1}{c_2}\left[\begin{array}{cc}
			1 \\[0.2cm]
			0\\
			\vdots\\
			0
		\end{array}\right]\phi_{m_1,h}. 	
	\end{align*}
	Then by \Cref{NonChang} and \eqref{c1c2} we have 
	\begin{align*}
		u_h^n  &= \sum_{i=1}^\ell\alpha_i^n
		\phi_{m_i,h}= \alpha_1^n
		\phi_{m_1,h} + \sum_{i=2}^\ell\alpha_i^n
		\phi_{m_i,h}\\
		& = \alpha_1^n
		\phi_{m_1,h}  + \sum_{i=2}^\ell\alpha_i^n \sum_{j=1}^{\ell-1} S^{-1}_{i-1,j}	\mathfrak{u}_h^j - \dfrac{c_1}{c_2}\alpha_2^n	\phi_{m_1,h}\\
		& =  \sum_{i=2}^\ell\alpha_i^n \sum_{j=1}^{\ell-1} S^{-1}_{i-1,j}	\mathfrak{u}_h^j. 
	\end{align*}
	This completes the proof.
	
\end{proof}

\begin{proof}[Proof of \Cref{lemma2}]
	First, we take  $v_h\in V_r\subset V_h$ in \eqref{discreteODE} to obtain
	\begin{align}\label{full_model_e}
		\left(\partial_t^+ u_h, v_{h}\right)+  a\left(u_h^n,  v_{h}\right) =\left(\sum_{i=1}^\ell c_i\phi_{m_i,h}, v_{h}\right). 	
	\end{align}
	Subtract \eqref{full_model_e} from \eqref{ROM} and we let $e^n = u_r^n - u_h^n$, then  
	\begin{align*}
		\left(\partial_t^+ e^n, v_{h}\right)+  a\left(e^n,  v_{h}\right)  =0.
	\end{align*}
	By \Cref{NonChang22}  we  take $v_h = e^n\in V_r$ and the identity
	\begin{align*}
		(a-b, a) & = \dfrac 1 2 (\|a\|^2 - \|b^2\|) + \dfrac 1 2 \|a-b\|^2,\\
		\frac{1}{2}(3 a-4 b+c,  a )&= 
		\frac{1}{4}\left[\|a\|^{2}+\|2 a-b\|^{2}-\|b\|^{2}-\|2 b-c\|^{2}\right]+\frac{1}{4}\|a-2 b+c\|^{2},
	\end{align*}
	to get
	\begin{align*}
		\|e^1\|^2 - \|e^0\|^2 + \left\|e^1-e^0\right\|^2 +2 \Delta t\left\| e^1 \right\|_V^2 = 0.
	\end{align*}
	Since $e^0 = 0$, then  $e^1=0$. In other words
	\begin{align*}
		u_r^1  = u_h^1.
	\end{align*}
	For $n\ge 2$ we have
	\begin{align*}
		\left[\|e^n\|^{2} - 		 \|e^{n-1}\|^{2}\right] +  \left[\|2e^n-e^{n-1}\|^{2} - 		 \|2e^{n-1}-e^{n-2}\|^{2}\right] + \|e^n-2 e^{n-1}+e^{n-2}\|^{2} + 4\Delta t \|e^n\|_V  = 0.
	\end{align*}
	Summing both sides of the above identity from $n=2$ to $n=m$ completes the proof of \Cref{lemma2}.
\end{proof}

\section{General data}\label{NEwway}
In this section, we extend the new algorithm to more general cases. We assume that the source term $f$ can be expressed or approximated  by a few only time dependent functions $f_{i}(t)$ and space dependent functions $g_i(x)$, i.e., 
\begin{align*}
	f(t,x) = \sum_{i=1}^m f_i(t) g_i(x),
\end{align*}
or
\begin{align*}
	f(t,x) \approx  \sum_{i=1}^m f(t_i^*,x)L_{k,i}(t):= \sum_{i=1}^m f_i(t) g_i(x),
\end{align*}
where $\{t_i^*\}_{i=1}^m$ are the $m$ Chebyshev interpolation nodes and $\{L_{m,i}(t)\}_{i=1}^m$ are the Lagrange interpolation functions:
\begin{align*}
	t_{i}^*&=\frac{T}{2}+\frac{T}{2} \cos \frac{(2 i-1) \pi}{2 m}\quad \textup{for}\quad i=1,2,\ldots, m,\\
	L_{m,i}(t) &= \frac{(t-t_1^*)\cdots(t-t_{i-1}^*)(t-t_{i+1}^*)\cdots (t-t_{m}^*)}{(t_i^*-t_1^*)\cdots(t_i^*-t_{i-1}^*)(t_i^*-t_{i+1}^*)\cdots (t_i^*-t_{m}^*)}.	
\end{align*}
Let $\{\varphi_i\}_{i=1}^N$ be the finite element basis function and we then define the following vectors:
\begin{align}\label{def_B}
	b_0 = [(u_0, \varphi_j)]_{j=1}^N,  \qquad
	b_i = [(g_i, \varphi_j)]_{j=1}^N, \;\; i=1,2\ldots, m,\qquad B = [b_0\mid b_1\mid b_2\mid,\ldots, \mid b_m].
\end{align}
\begin{algorithm}[H]
	\caption{}
	\label{D_matrixq3}
	{\bf{Input}:}  tol, $\ell$, $M$, $A$, $B$
	\begin{algorithmic}[1]
		\State Solve $A\texttt{u}_{1} = B$;
		\For{$i=2$ to $\ell$}
		\State Solve $A\texttt{u}_{i} = M \texttt{u}_{i-1}$;		
		\EndFor
		\State Set $U = [\texttt{u}_{1} \mid \texttt{u}_{2} \mid \ldots\mid \texttt{u}_\ell]$;
		\State Set $K = U^\top A U$;
		\State $[\Psi, \Lambda] = \textup{eig}(K)$;
		\State Find minimal $r$ such that $\sum_{k=1}^r {\Lambda(k,k)}/{\sum_{k=1}^{(i+1)\ell} \Lambda(k,k)}\ge 1-\textup{tol}$;
		\State Set $Y = U\Psi(:, 1:r)(\Lambda(1:r;1:r))^{-1/2}$;
		\State {\bf return} $Y$
	\end{algorithmic}
\end{algorithm}

By the linearity, one trivial idea is to apply \Cref{algorithm3} $m+1$ times. This can be computationally expensive if $m$ is not small. Following an ensemble idea in \cite{MR4027849}, we can treat $\{b_i\}_{i=0}^m$ simultaneously (see \Cref{D_matrixq3}) since these linear systems share a common coefficient matrix; it is more  efficient than solving the linear system with a single RHS for $m+1$ times. 

Unfortunately, a downside of this approach is a loss of numerical precision. More precisely, the convergence rates are not stable when the mesh size $h$ and time step $\Delta t$ are small, or when the error is pretty small; see \Cref{table_2}. This is possibly due to the fast decay of the eigenvalues of $K$. To resolve the problem, we use the singular values of $U$ as the stopping criteria. We see that the convergence is very stable; see \Cref{table_5}.

We note that the standard singular value decomposition (SVD) of $U$ is not equivalent to the eigen decomposition of $K$. In \cite{MR3775096}, the authors showed that the output of \Cref{D_matrixq3}, $Y$, is the first $r$ columns of $Q$ in the following \Cref{def_CSVD}.
\begin{definition}\label{def_CSVD} \textup{\cite{MR3775096}}
	A core SVD of a matrix $U\in \mathbb R^{N\times n}$ is a decomposition $U=Q \Sigma R^{\top}$, where $Q \in \mathbb{R}^{N \times d}, \Sigma \in \mathbb{R}^{d \times d}$, and $R \in \mathbb{R}^{n \times d}$ satisfy
	$$
	Q^{\top} A Q=I, \quad R^{\top} R=I, \quad \Sigma=\operatorname{diag}\left(\sigma_{1}, \ldots, \sigma_{d}\right),
	$$
	where $\sigma_{1} \geq \sigma_{2} \geq \cdots \geq \sigma_{d}>0$. The values $\left\{\sigma_{i}\right\}$ are called the (positive) singular values of $U$ and the columns of $Q$ and $R$ are called the corresponding singular vectors of $U$.
\end{definition}

Next, we introduce an efficient method to compute the core SVD of $U$.
\subsection{Incremental SVD}
Incremental SVD was proposed by Brand in \cite{brand2002incremental}, the algorithm updates the SVD of a matrix when one or more columns are added to the matrix.  However, if we directly apply Brand's algorithm to compute the core SVD of $U$, we have to compute the Cholesky factorization of the weighted matrix $A$.  Recently, Fareed et al. \cite{MR3775096} modified Brand's algorithm in a weighted norm setting without computing the Cholesky factorization of $A$. 

Next we briefly discuss the main idea  of the  algorithm. 

Suppose we already have the rank-$k$ truncated core SVD of $U_i$ (the first $i$ columns of $U$):  
\begin{align}\label{SVD01}
	U_i = Q \Sigma R^{\top},\quad \textup{with}\;\; 	Q^{\top} A Q=I, \quad R^{\top} R=I, \quad \Sigma=\operatorname{diag}\left(\sigma_{1}, \ldots, \sigma_{k}\right),
\end{align}
where $\Sigma \in \mathbb{R}^{k \times k}$ is a diagonal matrix with the $k$ (ordered) singular values of $U_i$ on the diagonal, $Q \in \mathbb{R}^{N \times k}$ is the matrix containing the corresponding $k$ left singular vectors of $U_i$, and $R \in \mathbb{R}^{i \times k}$ is the matrix of the corresponding $k$ right singular vectors of $U_i$.

The main idea of Brand's algorithm is to update the matrix $U_{i+1}:=	\left[\begin{array}{ll}
	U_i& 	u_{i+1}
\end{array}\right]$ by using the SVD of $U_i$ and the new coming data $u_{i+1}$. Set $p = \|u_{i+1} - QQ^\top A u_{i+1}\|_A$, where $\|x\|_A^2 = x^\top A x $. Then the  incremental SVD algorithm arises from the following fundamental identity:
\begin{tcolorbox}
	\begin{align}\label{fundamental_increment_SVD}
		\begin{split}
			U_{i+1} =
			\left[\begin{array}{ll}
				Q &  (I - QQ^\top A)u_{i+1}/p
			\end{array}\right]\left[\begin{array}{cc}
				\Sigma & Q^{\top}A u_{i+1} \\
				0 & p
			\end{array}\right] \left[\begin{array}{cc}
				R & 0 \\
				0 & 1
			\end{array}\right]^{\top}.	
		\end{split}
	\end{align}
\end{tcolorbox}
Then the SVD of $U_{i+1}$ can be constructed by
\begin{itemize}
	
	\item [1.] finding the full SVD of $\left[\begin{array}{cc}
		\Sigma & Q^\top A u_{i+1} \\
		0 & p
	\end{array}\right] = \widetilde  Q \widetilde \Sigma \widetilde R^\top$, and then
	\item[2.] updating the SVD of $U_{i+1}$ by
	$$
	U_{i+1}=\left(	\left[\begin{array}{ll}
		Q &  (I - QQ^\top A)u_{i+1}/p
	\end{array}\right] \widetilde Q\right) \widetilde \Sigma \left(\left[\begin{array}{cc}
		R & 0 \\
		0 & 1
	\end{array}\right] \widetilde R \right)^{\top}.
	$$
	
\end{itemize}

In practice, small numerical errors cause a loss of orthogonality, hence   reorthogonalizations are needed and the computational cost can be high. Very recently,  Zhang \cite{Zhang2022isvd} proposed an efficient way to reduce the complexity; see \cite[Section 4.2]{Zhang2022isvd} for more details.

Now we  apply the new incremental SVD algorithm in \cite{Zhang2022isvd} to find the core SVD of the matrix $U$ in \Cref{D_matrixq3}. Obviously, the  main computational cost is to solve the matrix equations for $\ell$ times.  Although solving a linear system with multiple RHS is more efficient than solving multiple linear systems with a single RHS, it is still expensive if we have a large amount of RHS. To reduce the computational cost, we  find a low rank approximation of the matrix $B$ in \eqref{def_B}.

Let $B = Q_F\Sigma_F R_F^\top$ be a thin SVD\footnote{In Matlab, we use \texttt{svd}$(B$, {\color{magenta} `econ'}) to compute the thin SVD of $B$, when $B$ has a large number of columns, we recommend the incremental SVD.} of $B$, the number of the columns of $Q_F$ is small if the data is low rank or approximately low rank. We summarize the above discussions  in  \Cref{algorithm4}.
\begin{algorithm}[H]
	\caption{}
	\label{algorithm4}
	{\bf{Input}:}  $M\in \mathbb R^{N\times N}$, $A\in \mathbb R^{N\times N}$, $ \{b_i\}_{i=0}^m$, \texttt{tol}
	\begin{algorithmic}[1]
		\State 	 $[Q_F, \Sigma_F, R_F]=$ \texttt{svd}$(B$, {\color{magenta} `econ'});
		\State Find the minimal $p$ such that $\Sigma_F(p+1,p+1)\le \texttt{tol}$;
		\State Let $Q_p$ be the first $p$ columns of $Q_F$;
		\State Solve  $	A\texttt{U}_{1} = Q_p$;
		\State $Q=[]$, $\Sigma=[]$, $R=[]$;
		\State 	 $[Q, \Sigma, R]=$ \texttt{Fullisvd}($Q, \Sigma, R, A$, $\texttt{U}_{1}$, \texttt{tol});  \hspace{3.5 cm} \% \cite[Algorithm 9]{Zhang2022isvd}
		
		\For{$i=2$ to $\ell$}
		\State Solve $ A\texttt{U}_{i} = M\texttt{U}_{i-1}$;	
		\State $[Q, \Sigma, R] =$ \texttt{Fullisvd}($Q, \Sigma, R, A$, $\texttt{U}_{i}$, \texttt{tol}); \hspace{2.95 cm} \% \cite[Algorithm 9]{Zhang2022isvd}
		\EndFor

		\State {\bf return} $Q$, $Q_F$, $\Sigma_F$, $R_F$
	\end{algorithmic}
\end{algorithm}

\subsection{ROM}
Let $r$ be the number of the column of $Q$ in \Cref{algorithm4}, and  define $V_r$ by
\begin{align}\label{NewPODbasis2}
	V_r = \textup{span}\{\widetilde \varphi_1, \widetilde\varphi_2,\ldots, \widetilde\varphi_{r}\}, \quad \textup{with}\quad  \widetilde \varphi_j = \sum_{k=1}^N Q(k,j)  \varphi_k,\quad \textup{for } j=1,2\cdots, r.
\end{align}

We are looking  for a function $u_r (t) \in C((0, T]; V_r)$ satisfying
\begin{align}\label{New_POD_reduced2}
	\begin{split}
		\frac{\mathrm{d}}{\mathrm{d} t} (u_{r}(t), v_h)+a(u_r(t), v_h) &= (f, v_h)\qquad\qquad\;\;\;\; \forall v_h\in V_r, \\
		(u_{r}(0), v_h) &=(u_0, v_h)\quad\quad\quad\quad\quad\forall v_h\in V_r.
	\end{split}
\end{align}
Since $u_r(t) \in V_r$ holds, we make the Galerkin ansatz of the form
\begin{align}\label{sec4_coe}
	u_r(t)=\sum_{j=1}^{r} u_{j}^{r}(t) \widetilde \varphi_j .	
\end{align}
We insert \eqref{sec4_coe} into \eqref{New_POD_reduced2} and define the modal coefficient vector
$$
U_{r}(t)=\left(u_{j}^{r}(t)\right)_{1 \leq j \leq r} \quad \text { for } t \in[0, T].
$$
From \eqref{New_POD_reduced2} we derive the linear system of ordinary differential equations
\begin{align}\label{semidiscrete4}
	\begin{split}
		M_{r}{U}'_{r}(t)+A_{r} U_{r}(t)&=\sum_{i=1}^m f_{i}(t) b_i^r\quad t \in (0, T], \\
		U_{r}(0)&= b_0^r.		
	\end{split}
\end{align}
where
\begin{align*}
	M_r = Q^\top M Q  \in \mathbb R^{r\times r}, \;\; 	A_r = Q^\top A Q \in \mathbb R^{r\times r}, \;\;  b_i^r =  Q^\top b_i  \in \mathbb R^{r},  i=0,1,2,\ldots, k.
\end{align*}
Note that \eqref{semidiscrete4} can then be solved by using an appropriate method for the time discretization.	

\subsection{Numerical experiments} 
In this section, we present several numerical tests to show the accuracy and efficiency of our new method. In all examples, we let  $\Omega = (0,1)^2$ in 2D and $\Omega = (0,1)^3$ in 3D and final time $T=1$.	The source term $f$ is chosen so that the exact solution is 
\begin{align*}
	u = \begin{cases}
		\sin(t) \cos(tx) x\sin(x - 1)\sin(y)(y - 1)& \textup{in} \;\;2D,\\
		\sin(t) \cos(tx) x\sin(x - 1)\sin(y)(y - 1)\sin(z)(z - 1)& \textup{in} \;\; 3D.
	\end{cases}
\end{align*}

First, we choose $\ell=5, m = 8$, $\texttt{tol} = 10^{-10}$ and  apply \Cref{D_matrixq3} to get the ROM \eqref{semidiscrete4}, then we use  the backward Euler for the first step and then apply the two-steps backward differentiation formula (BDF2) for the time discretization and take time step $\Delta t = h^{(k+1)/2}$, here $h$ is the mesh size and $k$ is the polynomial degree.  We report the error at the final time $T=1$ and wall time  in \Cref{table_2}. We see that the convergence rate of the ROM is the same as the standard finite element method when $k=1$. However,  the convergence rates are not stable due to a loss of numerical precision; see the case of $k=2$ in  \Cref{table_2}.
\begin{table}[H]
	\centering
	{
		\begin{tabular}{c|c|c|c|c|c|cc}
			\Xhline{1pt}

			\multirow{2}{*}{Degree}
			&\multirow{2}{*}{$\frac{h}{\sqrt{2}}$}
			&\multirow{2}{*}{Wall time (s)}	
			&\multicolumn{2}{c|}{$\|u - u_r\|_{L^2(\Omega)}$}	
			&\multicolumn{2}{c}{$\|\nabla(u-u_r)\|_{L^2(\Omega)}$}	
			\\
			\cline{4-7}
			& &&Error &Rate			&Error &Rate
			\\
			\cline{1-7}
			\multirow{5}{*}{ $k=1$}
			&	$2^{-3}$	&	 0.1876		&9.3278E-04	&		-	&	 	 1.9653E-02			&	-	&	  \\ 
			&	$2^{-4}$	&	 0.0570		&2.3677E-04	&		1.98	&	 	 9.8978E-03			&	0.99	&	  \\ 
			&	$2^{-5}$	&	 0.0922		&5.9423E-05	&		1.99	&	 	 4.9579E-03			&	1.00	&	  \\ 
			&	$2^{-6}$	&	 0.2523		&1.4873E-05	&		2.00	&	 	 2.4801E-03			&	1.00	&	  \\ 
			&	$2^{-7}$	&	 1.0610		&3.7233E-06	&		2.00	&	 	 1.2402E-03			&	1.00	&	  \\

			\cline{1-7}
			\multirow{5}{*}{ $k=2$}
			&	$2^{-3}$	&	0.1552& 		2.3143E-05	&		-	&	 	     1.4701E-03			&	-	&	  \\ 
			&	$2^{-4}$	&	0.0943& 		2.8966E-06	&		3.00	&	 	 3.7004E-04			&	1.99	&	  \\ 
			&	$2^{-5}$	&	0.2106& 		3.7517E-07	&		2.95	&	 	 9.2685E-05			&	2.00	&	  \\ 
			&	$2^{-6}$	&	0.7487& 		1.0784E-07	&		1.80	&	 	 2.3222E-05			&	2.00	&	  \\ 
			&	$2^{-7}$	&	3.4102& 		9.7986E-08	&		0.14	&	 	 5.9637E-06			&	1.96	&	   \\

			\Xhline{1pt}

		\end{tabular}
	}
	\caption{\Cref{D_matrixq3}: 2D: The errors for $u_r$ and $\nabla u_r$ at the final time $T=1$.}\label{table_2}
\end{table}

Second, we test the efficiency and accuracy of \Cref{algorithm4} under the same problem data and setting as above. We report the error at the final time $T=1$ and wall time  in \Cref{table_4,table_5}. We see that the convergence rate of the ROM is the same as the standard finite element method. 

\begin{table}[H]
	\centering
	{
		\begin{tabular}{c|c|c|c|c|c|cc}
			\Xhline{1pt}

			\multirow{2}{*}{Degree}
			&\multirow{2}{*}{$\frac{h}{\sqrt{2}}$}
			&\multirow{2}{*}{Wall time (s)}	
			&\multicolumn{2}{c|}{$\|u - u_r\|_{L^2(\Omega)}$}	
			&\multicolumn{2}{c}{$\|\nabla(u-u_r)\|_{L^2(\Omega)}$}	
			\\
			\cline{4-7}
			& &&Error &Rate			&Error &Rate
			\\
			\cline{1-7}
			\multirow{5}{*}{ $k=1$}
			&	$2^{-3}$	&	 0.1438		&9.3549E-04&		-	    &	 	 1.9653E-02			&	-	    &	  \\ 
			&	$2^{-4}$	&	 0.0473		&2.3727E-04&		1.98	&	 	 9.8978E-03			&	0.99	&	  \\ 
			&	$2^{-5}$	&	 0.0779		& 5.9521E-05	&	2.00	&	 	 4.9579E-03			&	1.00	&	  \\ 
			&	$2^{-6}$	&	 0.2351		&1.4891E-05	&		2.00	&	 	 2.4801E-03			&	1.00	&	  \\ 
			&	$2^{-7}$	&	 1.0067		&3.7232E-06	&		2.00	&	 	 1.2402E-03			&	1.00	&	  \\

			\cline{1-7}
			\multirow{5}{*}{ $k=2$}
			&	$2^{-3}$	&	0.1893& 		2.3208E-05	&		-	&	 	     1.4701E-03			&	-	&	  \\ 
			&	$2^{-4}$	&	0.1027& 		2.8977E-06	&		3.00	&	 	 3.7004E-04			&	1.99	&	  \\ 
			&	$2^{-5}$	&	0.2585& 	3.6209E-07	&		3.00	&	 	 9.2674E-05			&	2.00	&	  \\ 
			&	$2^{-6}$	&	0.8738& 		4.5252E-08	&		3.00	&	 	 2.3179E-05			&	2.00	&	  \\ 
			&	$2^{-7}$	&	3.8290& 		5.6561E-09	&		3.00	&	 	 5.7954E-06			&	2.00	&	  \\

			\Xhline{1pt}

		\end{tabular}
	}
	\caption{\Cref{algorithm4}: 2D: The errors for $u_r$ and $\nabla u_r$ at the final time $T=1$.}\label{table_4}
\end{table}

\begin{table}[H]
	\centering
	{
		\begin{tabular}{c|c|c|c|c|c|cc}
			\Xhline{1pt}

			\multirow{2}{*}{Degree}
			&\multirow{2}{*}{$\frac{h}{\sqrt{2}}$}
			&\multirow{2}{*}{Wall time (s)}	
			&\multicolumn{2}{c|}{$\|u - u_r\|_{L^2(\Omega)}$}	
			&\multicolumn{2}{c}{$\|\nabla(u-u_r)\|_{L^2(\Omega)}$}	
			\\
			\cline{4-7}
			& &&Error &Rate			&Error &Rate
			\\
			\cline{1-7}
			\multirow{6}{*}{ $k=1$}
			&	$2^{-2}$	&	 0.2370		&9.5386E-04	&		-	&	 	 1.0151E-02			&	-	&	  \\ 
			&	$2^{-3}$	&	 0.2663		&2.6633E-04	&		1.84	&	 	 5.3315E-03			&	0.93	&	  \\ 
			&	$2^{-4}$	&	 2.8694		&6.8562E-05	&		1.96	&	 	 2.7000E-03			&	0.98	&	  \\ 
			&	$2^{-5}$	&	 15.693		&1.7269E-05	&		1.99	&	 	 1.3544E-03			&	1.00	&	  \\ 
			&	$2^{-6}$	&	 145.73		&4.3254E-06	&		2.00	&	 	 6.7774E-04			&	1.00	&	  \\

			\cline{1-7}
			\multirow{5}{*}{ $k=2$}
			&	$2^{-2}$	&	0.3842& 		2.3139E-05	&		-	&	 	 1.4701E-03			&	-	&	  \\ 
			&	$2^{-3}$	&	1.8871& 		7.3028E-06	&		3.00	&	 	 4.9435e-04			&	1.99	&	  \\ 
			&	$2^{-4}$	&	11.790& 		9.0664E-07	&		3.00	&	 	 1.2573E-04			&	2.00	&	  \\ 
			&	$2^{-5}$	&	120.47& 		1.1318E-07	&		3.00	&	 	 3.1585E-05			&	2.00	&	  \\ 
			&	$2^{-6}$	&	1294.7& 		1.4144E-08	&		3.00	&	 	 7.9064E-06			&	2.00	&	  \\

			\Xhline{1pt}

		\end{tabular}
	}
	\caption{\Cref{algorithm4}: 3D: The errors for $u_r$ and $\nabla u_r$ at the final time $T=1$.}\label{table_5}
\end{table}

\section{Conclusion} In the paper, we proposed a new reduced order model (ROM) of linear parabolic PDEs. We proved that the singular values  of the Krylov sequence are exponential decaying (\Cref{UAU}). Furthermore, under some assumptions, we proved that the solutions of the ROM and the FEM have the same convergence rates (\Cref{Main_res}).  There are many interesting directions for the future research. First, we will investigate the non-homogeneous Dirichlet boundary conditions and the corresponding Dirichlet boundary control problems, such as \cite{MR4381532,MR4169689,MR4057428,MR3992054,MR3831243}.
Second, we will consider the ROM for hyperbolic PDEs and Maxwell's equations \cite{monk2019finite,MR2059447}.  Third, we will  apply our result for realistic problems, such as  inverse problems,  shape optimization problems and data assimilation problems. Our long term goal is to build  accurate ROMs for nonlinear PDEs.

\bibliographystyle{siamplain}
\bibliography{Model_Order_Reduction}

\begin{thebibliography}{10}

\bibitem{bai2020deim}
{\sc F.~Bai and Y.~Wang}, {\em Deim reduced order model constructed by hybrid
  snapshot simulation}, SN Applied Sciences, 2 (2020), pp.~1--25.

\bibitem{MR3947283}
{\sc B.~Beckermann and A.~Townsend}, {\em Bounds on the singular values of
  matrices with displacement structure}, SIAM Rev., 61 (2019), pp.~319--344,
  \url{https://doi.org/10.1137/19M1244433}.
\newblock Revised reprint of ``On the singular values of matrices with
  displacement structure" [ MR3717820].

\bibitem{MR3419868}
{\sc P.~Benner, S.~Gugercin, and K.~Willcox}, {\em A survey of projection-based
  model reduction methods for parametric dynamical systems}, SIAM Rev., 57
  (2015), pp.~483--531, \url{https://doi.org/10.1137/130932715}.

\bibitem{brand2002incremental}
{\sc M.~Brand}, {\em Incremental singular value decomposition of uncertain data
  with missing values}, in European Conference on Computer Vision, Springer,
  2002, pp.~707--720.

\bibitem{Chapelle_Gariah_Sainte_Galerkin_M2AN_2013}
{\sc D.~Chapelle, A.~Gariah, P.~Moireau, and J.~Sainte-Marie}, {\em A
  {G}alerkin strategy with proper orthogonal decomposition for
  parameter-dependent problems---analysis, assessments and applications to
  parameter estimation}, ESAIM Math. Model. Numer. Anal., 47 (2013),
  pp.~1821--1843, \url{https://doi.org/10.1051/m2an/2013090}.

\bibitem{MR4027849}
{\sc G.~Chen, L.~Pi, L.~Xu, and Y.~Zhang}, {\em A superconvergent ensemble
  {HDG} method for parameterized convection diffusion equations}, SIAM J.
  Numer. Anal., 57 (2019), pp.~2551--2578,
  \url{https://doi.org/10.1137/18M1192573}.

\bibitem{MR3992054}
{\sc G.~Chen, J.~R. Singler, and Y.~Zhang}, {\em An {HDG} method for
  {D}irichlet boundary control of convection dominated diffusion {PDE}s}, SIAM
  J. Numer. Anal., 57 (2019), pp.~1919--1946,
  \url{https://doi.org/10.1137/18M1208708}.

\bibitem{MR3775096}
{\sc H.~Fareed, J.~R. Singler, Y.~Zhang, and J.~Shen}, {\em Incremental proper
  orthogonal decomposition for {PDE} simulation data}, Comput. Math. Appl., 75
  (2018), pp.~1942--1960, \url{https://doi.org/10.1016/j.camwa.2017.09.012}.

\bibitem{MR1273642}
{\sc D.~Fasino and G.~Inglese}, {\em On the spectral condition of rectangular
  {V}andermonde matrices}, Calcolo, 29 (1992), pp.~291--300 (1993),
  \url{https://doi.org/10.1007/BF02576186}.

\bibitem{MR3697034}
{\sc X.~Fu and J.~N. Kutz}, {\em Adaptive dimensionality-reduction for
  time-stepping in differential and partial differential equations}, Numer.
  Math. Theory Methods Appl., 10 (2017), pp.~872--894,
  \url{https://doi.org/10.4208/nmtma.2017.m1624}.

\bibitem{MR3831243}
{\sc W.~Gong, W.~Hu, M.~Mateos, J.~Singler, X.~Zhang, and Y.~Zhang}, {\em A new
  {HDG} method for {D}irichlet boundary control of convection diffusion {PDE}s
  {II}: low regularity}, SIAM J. Numer. Anal., 56 (2018), pp.~2262--2287,
  \url{https://doi.org/10.1137/17M1152103}.

\bibitem{MR4169689}
{\sc W.~Gong, W.~Hu, M.~Mateos, J.~R. Singler, and Y.~Zhang}, {\em Analysis of
  a hybridizable discontinuous {G}alerkin scheme for the tangential control of
  the {S}tokes system}, ESAIM Math. Model. Numer. Anal., 54 (2020),
  pp.~2229--2264, \url{https://doi.org/10.1051/m2an/2020015}.

\bibitem{MR4381532}
{\sc W.~Gong, M.~Mateos, J.~Singler, and Y.~Zhang}, {\em Analysis and
  approximations of {D}irichlet boundary control of {S}tokes flows in the
  energy space}, SIAM J. Numer. Anal., 60 (2022), pp.~450--474,
  \url{https://doi.org/10.1137/21M1406799},
  \url{https://doi.org/10.1137/21M1406799}.

\bibitem{gubisch2017proper}
{\sc M.~Gubisch and S.~Volkwein}, {\em Proper orthogonal decomposition for
  linear-quadratic optimal control}, Model reduction and approximation: theory
  and algorithms, 15 (2017).

\bibitem{MR2327057}
{\sc C.~Homescu, L.~R. Petzold, and R.~Serban}, {\em Error estimation for
  reduced-order models of dynamical systems}, SIAM Rev., 49 (2007),
  pp.~277--299, \url{https://doi.org/10.1137/070684392}.

\bibitem{MR4057428}
{\sc W.~Hu, J.~Shen, J.~R. Singler, Y.~Zhang, and X.~Zheng}, {\em A
  superconvergent hybridizable discontinuous {G}alerkin method for {D}irichlet
  boundary control of elliptic {PDE}s}, Numer. Math., 144 (2020), pp.~375--411,
  \url{https://doi.org/10.1007/s00211-019-01090-2}.

\bibitem{MR4296762}
{\sc B.~Koc, S.~Rubino, M.~Schneier, J.~Singler, and T.~Iliescu}, {\em On
  optimal pointwise in time error bounds and difference quotients for the
  proper orthogonal decomposition}, SIAM J. Numer. Anal., 59 (2021),
  pp.~2163--2196, \url{https://doi.org/10.1137/20M1371798}.

\bibitem{Kunisch_Volkwein_NM_2001}
{\sc K.~Kunisch and S.~Volkwein}, {\em Galerkin proper orthogonal decomposition
  methods for parabolic problems}, Numer. Math., 90 (2001), pp.~117--148,
  \url{https://doi.org/10.1007/s002110100282}.

\bibitem{Kunisch_Volkwein_SINUM_2002}
{\sc K.~Kunisch and S.~Volkwein}, {\em Galerkin proper orthogonal decomposition
  methods for a general equation in fluid dynamics}, SIAM J. Numer. Anal., 40
  (2002), pp.~492--515, \url{https://doi.org/10.1137/S0036142900382612}.

\bibitem{MR4172732}
{\sc S.~Locke and J.~Singler}, {\em New proper orthogonal decomposition
  approximation theory for {PDE} solution data}, SIAM J. Numer. Anal., 58
  (2020), pp.~3251--3285, \url{https://doi.org/10.1137/19M1297002}.

\bibitem{MR2059447}
{\sc P.~Monk}, {\em Finite element methods for {M}axwell's equations},
  Numerical Mathematics and Scientific Computation, Oxford University Press,
  New York, 2003,
  \url{https://doi.org/10.1093/acprof:oso/9780198508885.001.0001}.

\bibitem{monk2019finite}
{\sc P.~Monk and Y.~Zhang}, {\em Finite element methods for maxwell's
  equations},  (2019).

\bibitem{MR2595807}
{\sc M.-L. Rap\'{u}n and J.~M. Vega}, {\em Reduced order models based on local
  {POD} plus {G}alerkin projection}, J. Comput. Phys., 229 (2010),
  pp.~3046--3063, \url{https://doi.org/10.1016/j.jcp.2009.12.029}.

\bibitem{MR972756}
{\sc J.~W. Ruge and K.~St\"{u}ben}, {\em Algebraic multigrid}, in Multigrid
  methods, vol.~3 of Frontiers Appl. Math., SIAM, Philadelphia, PA, 1987,
  pp.~73--130.

\bibitem{MR3987425}
{\sc J.~Shen, J.~R. Singler, and Y.~Zhang}, {\em H{DG}-{POD} reduced order
  model of the heat equation}, J. Comput. Appl. Math., 362 (2019),
  pp.~663--679, \url{https://doi.org/10.1016/j.cam.2018.09.031}.

\bibitem{Singler_New_SINUM_2014}
{\sc J.~R. Singler}, {\em New {POD} error expressions, error bounds, and
  asymptotic results for reduced order models of parabolic {PDE}s}, SIAM J.
  Numer. Anal., 52 (2014), pp.~852--876,
  \url{https://doi.org/10.1137/120886947}.

\bibitem{MR909407}
{\sc G.~Talenti}, {\em Recovering a function from a finite number of moments},
  Inverse Problems, 3 (1987), pp.~501--517.

\bibitem{MR2873253}
{\sc F.~Terragni, E.~Valero, and J.~M. Vega}, {\em Local {POD} plus {G}alerkin
  projection in the unsteady lid-driven cavity problem}, SIAM J. Sci. Comput.,
  33 (2011), pp.~3538--3561, \url{https://doi.org/10.1137/100816006}.

\bibitem{MR0271762}
{\sc H.~S. Wilf}, {\em Finite sections of some classical inequalities},
  Ergebnisse der Mathematik und ihrer Grenzgebiete, Band 52, Springer-Verlag,
  New York-Berlin, 1970.

\bibitem{Zhang2022isvd}
{\sc Y.~Zhang}, {\em An answer to an open question in the incremental {SVD}},
  \url{https://doi.org/https://arxiv.org/abs/2204.05398}.

\end{thebibliography}

\end{document}